\documentclass{siamart1116}

\usepackage{amsmath,amssymb}
\usepackage{cite}
\usepackage{mathtools}
\usepackage{tikz}
\usepackage{pgfplots}
\usepackage{pgfplotstable}
\usepackage{geometry}
\usepackage{color}
\usepackage{booktabs}
\usepackage{mathrsfs}
\usepackage{cleveref}
\usepackage[utf8]{inputenc}
\usepackage{enumitem}
\usepackage{algorithm}
\usepackage{algpseudocode}
\newcommand\TheTitle{A low-rank technique for computing the quasi-stationary distribution of subcritical Galton-Watson processes}
\newcommand\TheShortTitle{A low-rank technique for computing the QSD of Galton-Watson processes}
\newcommand\TheAuthors{Sophie Hautphenne and Stefano Massei}

\newcommand{\vc}[1]{\boldsymbol{#1}}

\headers{\TheShortTitle}{\TheAuthors}

\title{\TheTitle}

\author{	Sophie Hautphenne\thanks{University of Melbourne, Australia and EPF Lausanne, Switzerland
		\email{sophiemh@unimelb.edu.au}} \and
	Stefano Massei\thanks{EPF Lausanne, Switzerland,
		\email{stefano.massei@epfl.ch}}}

\pgfplotsset{compat=1.9}

\pgfplotstableset{
	every head row/.style={before row=\toprule,after row=\midrule},
	clear infinite
}

\pgfplotsset{select coords between index/.style 2 args={
		x filter/.code={
			\ifnum\coordindex<#1\fi
			\ifnum\coordindex>#2\fi
		}
}}

\usepackage{amsopn}
\DeclarePairedDelimiter{\norm}{\lVert}{\rVert}

\DeclareMathOperator{\diag}{diag}

\numberwithin{theorem}{section}

\newsiamremark{remark}{Remark}
\newsiamremark{example}{Example}

\renewcommand{\leq}{\leqslant}
\renewcommand{\geq}{\geqslant}

\ifpdf
\hypersetup{
	pdftitle={\TheTitle},
	pdfauthor={\TheAuthors}
}
\fi

\begin{document}
	\maketitle
	\begin{abstract}
		%Galton–Watson branching processes are classical tools for modeling the evolution of the size of a population in time. 
		We present  a new algorithm for computing the quasi-stationary distribution of subcritical Galton--Watson branching processes.
		This algorithm is based on a particular discretization of a well-known functional equation that characterizes the quasi-stationary distribution of these processes. We provide a theoretical analysis of the approximate low-rank structure that stems from this discretization, and we extend the procedure to multitype branching processes.
		We use numerical examples to demonstrate that our algorithm is both more accurate and more efficient than other
		approaches.
		\bigskip
		
		
		\bigskip 
		
	\end{abstract}
	\section{Introduction}
	Many biological populations are doomed to extinction due to low reproduction rates, the presence of predators, competition for limited resources, lack of suitable habitat, or other factors. However, before extinction eventually occurs the population size may fluctuate around some positive values for a long period of time. We are then interested in the long-term distribution of the size of the population; roughly speaking, this amounts to studying the \emph{quasi-stationary distribution}.  We illustrate this in Figure \ref{QSDex1} for a stochastic process with logistic growth. 
	%Precise definitions will be given in the next section.
	%The distribution of the population size during the stationary regime before extinction is called the \emph{quasi-stationary distribution} or \emph{quasi-limiting distribution}. Precise definitions will be given in the next section.
	
	\begin{figure}
		\centering
		\includegraphics[width=10cm]{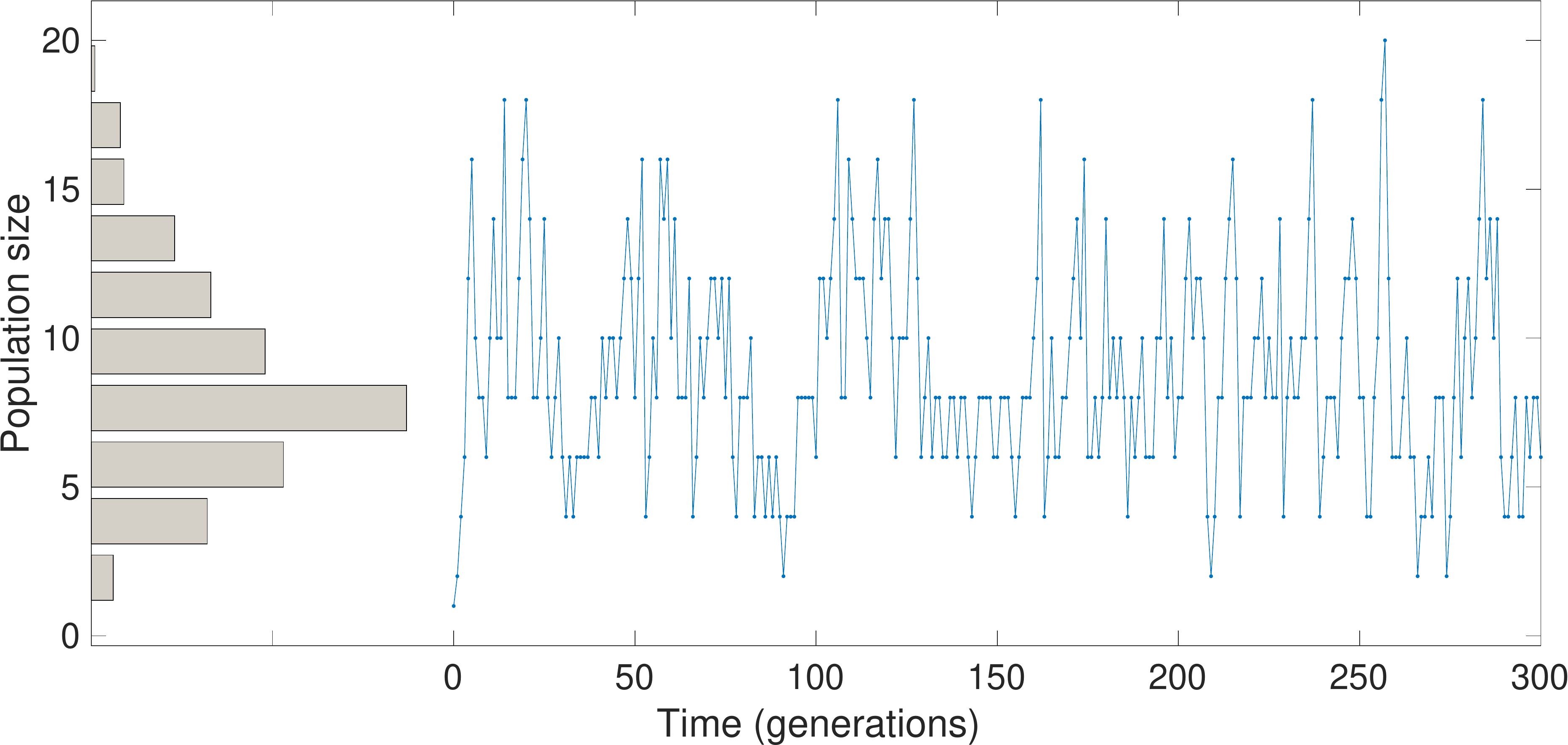}
		\caption{\label{QSDex1}A trajectory of a discrete-time population-size-dependent branching process, starting with a single individual. An empirical estimate of the quasi-stationary distribution is superimposed.}
	\end{figure}
	In this paper, we focus on the class of stochastic processes called the \emph{Galton–Watson} (GW) \emph{branching processes}. These processes are particular discrete-time Markov chains used to model randomly evolving populations in which individuals reproduce independently of each other. They have been
	successfully illuminating real-world problems that arise in diverse areas, such as biology, chemistry, particle
	physics, and computer science. Classical reference books on branching processes include Harris \cite{harris}, Athreya and Ney \cite{athreya2004branching}, and Haccou, Jagers and Vatutin \cite{haccou2005branching}.
	
	Quasi-stationary distributions of stochastic processes have been a focus of attention for many years. Their study started with the work of Yaglom in the late 1940's, who was the first to establish the existence of a particular quasi-stationary distribution, called the \emph{Yaglom limit}, in the (subcritical) GW {branching process} \cite{yaglom1947certain}. 
	The computation of quasi-stationary distributions of general Markov chains can generally be tackled from different angles. The most common approaches involve using simulation techniques, or solving for the left eigenvector of the transition matrix restricted to the positive states; for more details we refer to the excellent surveys of M\'el\'eard and Villemonais \cite{meleard2012quasi} and van Doorn and Pollet \cite{van2013quasi}, and references therein. These methods have clear limitations, especially when the state space of the process is unbounded. 
	%GW processes are particular discrete-time Markov chains used to model populations that evolve randomly in time, and will be our object of study here. They have been
	%successfully illuminating real-world problems arising in diverse areas, such as biology, chemistry, particle
	%physics, and computer science. Classical reference books on branching processes include Harris \cite{harris} and Athreya and Ney \cite{athreya2004branching}.  
	%The subcritical BGW branching process will be our processes of interest here.
	% For reasons that will be clear, the subcritical BGW processes will be the primary focus here
	%Here we focus on a specific class of Markov population models called branching processes. These processes are the primary tool used to model populations that evolve randomly in time. They have been
	%used successfully to illuminate real-world problems arising in many areas, such as biology, chemistry, particle
	%physics, and computer science. Key performance measures of branching processes include the extinction
	%probability of a population, the distribution of the population size at a given time, the time until extinction, and the
	%asymptotic population composition.
	Our motivation for considering subcritical GW processes here stems from the fact that their Yaglom limit has a specific characterisation: if $P(z):=\sum_{j\geq 0} p_j z^j$ denotes the (known) probability generating function of the offspring distribution, $m:=P'(1)<1$ its mean, and $G(z):=\sum_{j\geq 1} g_j z^j$ the unknown probability generating function of the Yaglom limit $(g_j)_{j\geq 1}$, then $G(z)$ solves the modified Schr\"oder functional equation
	\begin{equation}\label{eq:G}G(0)=0,\quad G(P(z))=m G(z)+1-m,\qquad z\in[0,1].\end{equation} There exists a unique probability generating function $G(z)$ that solves \eqref{eq:G} \cite[Theorem 1]{heathcote}.
	To the best of our knowledge, no attention has been paid to the numerical solution of this equation. 
	
	In this paper, we propose an efficient algorithmic method to compute the coefficients $g_j$ of $G(z)$ when the latter is analytic on a neighborhood of the unit disc. Our approach consists in using Cauchy's integral formula to rewrite \eqref{eq:G}  as
	\[
	\int_{\Gamma}\frac{G(t)}{t-P(z)}dt=m G(z)+1-m,
	\]
	where $\Gamma$ is a circle of appropriate radius $r$. Discretizing the integral on the left-hand side by means of the trapezoidal rule leads to
	\[
	\sum_{j=1}^{n} G(r\xi_j)\cdot \frac{r\xi_j}{n}\cdot (r\xi_j- P(z))^{-1}= mG(z)+1-m,
	\]where $\{r\xi_j\}_{1\leq j\leq n}$ are the scaled $n$th roots of unity.
	Evaluating the last equation in $z=r\xi_j$ for $1\leq j\leq n$ leads to a linear system where the unknowns are the (approximate) quantities $G(r\xi_j)$. We then retrieve an interpolating polynomial of degree $n$ for $G(z)$ by applying the Fast Fourier Transform (FFT). 
	This particular discretization method provides highly structured data and allows us to deal with a large number $n$ of integration nodes. In the second part of the paper we perform a theoretical analysis of the low-rank structure arising from the discretization scheme, and we discuss how to modify the algorithm in order to benefit from this property. The key idea is to approximate the coefficient matrix with a sum $mI-UV^*$ where $U,V$ are tall and skinny matrices retrieved by means of the \emph{Adaptive Cross Approximation} algorithm \cite[Algorithm 1]{borm}. This yields a procedure with almost linear complexity of computational time and storage.
	In the final part of the paper we extend the technique to multitype branching processes. Here, the computational cost and the memory consumption suffer from the curse of dimensionality. The presence of the low-rank structure enables to partially mitigate this effect and to obtain satisfactory results in the two-dimensional case.
	
	The paper is organized as follows; in Section~\ref{sec:back}
	%-\ref{sec:not} 
	we recall some background notions. We dedicate Section~\ref{sec:properties} to the study of the regularity of $G(z)$ and the consequent decay of the coefficients $g_j$ as $j\to\infty$. In particular, we provide results on the interplay between the regularity of $P(z)$ and $G(z)$. In Section~\ref{sec:alg} we describe the numerical procedures for the computation of the coefficients $g_j$. In Section~\ref{sec:ev-int} we introduce a new method based on solving a discretized version of Equation \eqref{eq:G}, and we compare it with other techniques in Section~\ref{sec:tests}. In Section~\ref{sec:rank} we perform an analysis of the rank structure stemming from the discretization process, and we provide a large scale version of the new algorithm in Section~\ref{sec:exploit}. In Section~\ref{sec:multi} we extend the procedure to multitype GW branching processes. Finally, we gather the proofs of the technical results of Section~\ref{sec:properties} in an appendix.
	
	Throughout the paper, for $z_0\in\mathbb C$ and $r>0$, we let $\mathcal B(z_0, r) :=\{z\in\mathbb C:\ |z-z_0|< r\} $, $\mathcal D(z_0, r) :=\{z\in\mathbb C:\ |z-z_0|\leq r\} $ and $\mathcal S^1 :=\{z\in\mathbb C:\ |z|= 1\}$; we let $\partial$ indicate the border of a set with respect to the Euclidean topology, e.g., $\mathcal S^1=\partial \mathcal B(0,1)=\partial \mathcal D(0,1)$; for $h\geq1$, we denote by $G^{(h)}(z_0)$ the $h$-th derivative of the function $G(z)$ evaluated at $z=z_0$;
	finally, we let $\vc 1$ and $\vc 0$ denote the column vectors of $1$'s and $0$'s, respectively, whose length will be determined by the context.
	
	% and we let $\vc e_i$ represent the vector with all entries equal to zero, except entry $i$ which is equal to 1, the size of these vectors being defined by the context.

	\subsection{Background}\label{sec:back}
	\bigskip
	
	A \emph{Galton--Watson} (GW) \emph{branching process} is a particular discrete-time Markov chain $\{Z_n\}_{n\geq 0}$ that takes its values in $\mathbb{N}:=\{0,1,2,\ldots\}$, where 0 is an absorbing state. It describes the evolution of a population in which each individual lives for one unit of time (generation), at the end of which it gives birth to a random (nonnegative integer) number $\theta$ of children, independently of the rest of the population. The distribution of $\theta$ is called the \emph{offspring distribution}, and is denoted by $\vc p:=(p_j)_{j\in \mathbb{N}}$, where $p_j:= \mathbb P(\theta=j)\geq 0,$ $ \sum_{j\geq 0} p_j=1$; more explicitly, $p_j$ represents the probability for an individual to have $j$ offspring at the end of her/his lifetime. For each $n\geq 0$, the random variable  $Z_n$ represents the number of individuals in the population at generation $n$. We define the \emph{probability generating function} of the offspring distribution as $$
	P(z):=\sum_{j\geq 0} p_j z^j,\quad z\in[0,1].$$ 
	The mean offspring number per individual in the GW process is given by $$m:=\mathbb E( \theta)=\frac{d P(z)}{dz}\Big|_{z=1} = \sum_{j\geq 1}j p_j.$$
	%$$P^{(j)}(z_0):=\left.\frac{d^j P(z)}{d z^j}\right|_{z=z_0}.$$
	Given the initial population size $Z_0$, the size $Z_{n}$ of the population at generation $n\geq 1$ evolves according to the recurrence formula
	$$Z_{n}=\sum_{i=1}^{Z_{n-1}} \theta_i^{(n)},$$where $\{ \theta_i^{(n)}\}_{i,n}$ is a family of independent random variables taking values in $\mathbb{N}$ with the same distribution as $\theta$, and $Z_n:=0$ if $Z_{n-1}=0$ (hence 0 is absorbing). Indeed, the population size at generation $n$, $Z_n$, is made of the children of the $Z_{n-1}$ individuals present in the population at generation $n-1$ (the random variable $\theta_i^{(n)}$ represents the number of children of the $i$th individual present in the population at generation $n-1$).
	Before extinction, the GW process takes its values in the space $\mathbb{N}_0:=\mathbb{N}\setminus\{0\}$. 
	
	In the sequel, we assume $0<p_0+p_1<1$. For any initial state $\ell\in \mathbb{N}_0$ and any initial probability distribution $\vc \mu:=(\mu_\ell)_{\ell\in \mathbb{N}_0}$, we let $\mathbb P_\ell(\cdot):=\mathbb P(\cdot \,|\,Z_0=\ell)$, and $\mathbb P_\mu(\cdot):=\sum_{\ell\geq 1}\mu_\ell \mathbb P_\ell(\cdot )$ be the probability measure conditional on the distribution of $Z_0$.

	We distinguish between three cases:
	\begin{itemize}
		\item the \emph{subcritical} case $m<1$: the population becomes extinct with probability one, that is, for any initial probability distribution $\vc \mu$, $\mathbb P_\mu(\exists n<\infty: Z_n=0)=1$; the expected extinction time is finite.
		\item the \emph{critical} case $m=1$: the population becomes extinct with probability one, and the expected extinction time is infinite.
		\item the \emph{supercritical} case $m>1$: the population has a positive probability of surviving, and therefore the expected extinction time is infinite.
	\end{itemize}We refer to \cite{harris, athreya2004branching,haccou2005branching} for basic properties of GW processes.
	
	%We will assume that

	We say that $\{Z_n\}$ has a \emph{Yaglom limit} if there exists a probability distribution $\vc g:=(g_j)_{j\in \mathbb{N}_0}$ (that is with $g_j\in [0,1]$ for all $j\geq1$, and  $\sum_{j\geq 1} g_j=1$) such that, \emph{for any initial population size} $\ell\in \mathbb{N}_0$ and any state $j\in \mathbb{N}_0$,
	\[
	\lim_{n\rightarrow\infty} \mathbb P_\ell(Z_n=j\,|\,Z_n>0)=g_j;
	\]in other words, $\vc g$ is the \emph{asymptotic} distribution of the population size at generation $n$, conditional on non-extinction by generation $n$.
	When it exists, the Yaglom limit $\vc g$ is a \emph{quasi-stationary distribution}, that is, 
	for all $n\geq 0$ and for any state $j\in \mathbb{N}_0,$
	\[
	\mathbb P_g(Z_n=j\,|\,Z_n>0)=g_j;
	\]in other words, if the process starts with a number of individuals distributed according to $\vc g$, then the distribution of the population size at any subsequent generation $n\geq 1$ remains $\vc g$.

	There is no quasi-stationary distribution in the critical and the supercritical case because conditioning on the event $\{Z_n>0\}$  results in the process growing without bounds as $n\rightarrow\infty$. However, in the subcritical case there is a unique Yaglom limit; the next theorem states this formally, and we refer to \cite[Theorem 6]{meleard2012quasi} for a proof.
	
	\begin{theorem}[Yaglom \cite{yaglom1947certain}] Let $\{Z_n\}$ be a GW process with offspring generating function $P(z)$ and mean offspring $m<1$. There exists a unique probability distribution $\vc g=(g_j)_{j\in \mathbb{N}_0}$ such that, for any initial probability distribution $\vc\mu$ with finite mean on $\mathbb{N}_0$ (that is, $\sum_{j\geq1}j \mu_j<\infty$), $\vc g$ satisfies
		\begin{equation}\label{qld}
		\lim_{n\rightarrow\infty} \mathbb P_\mu(Z_n=j\,|\,Z_n>0)=g_j.
		%\quad \mbox{for all } \vc\mu \mbox{ such that } E_\mu(Z_0)<%%\infty.
		\end{equation}
		The distribution $\vc g$ is a Yaglom limit for $\{Z_n\}$, and its generating function $G(z):=\sum_{j\geq 1} g_j z^j,$ $z\in[0,1],$ satisfies Equation \eqref{eq:G} on $[0,1]$.
		%\footnote{I removed the claim about uniqueness because I am not sure this is true if $G(z)$ is not analytic at $1$. The point is that given two solutions $G_1$ and $G_2$ their evaluations have to match on the sequence $z_{k+1}=P(z_k)$, $z_0=0$; this leads to $G_1\equiv G_2$ only if $1$ belongs to the domain of analyticity of $G_j$ $j=1,2$ because then the zeros of $G_1-G_2$ accumulates to something inside this domain.}
		%{\color{blue}solves}\footnote{Stefano: It is unique only if we include the boundary condition $G(1)=1$.} the equation 
		%\begin{equation}\label{FFP}
		%G(P(z))=mG(z)+1-m, \quad s\in[0,1].
		%\end{equation}
	\end{theorem}

	\begin{remark}
		Since the series that define $P(z)$ and $G(z)$ converge absolutely for all $z$ in the unit disc $\mathcal D(0,1)$ then, by continuity, \eqref{eq:G} holds $\forall z\in \mathcal D(0,1)$.
	\end{remark}
	\begin{remark}
		There exist quasi-stationary distributions which are not a Yaglom limit. In particular, for the subcritical GW process, there exists an infinite number of quasi-stationary distributions which are obtained as the limit in \eqref{qld} for some initial probability distributions $\vc\mu$ \emph{with infinite mean}. We refer to  \cite[Theorem 6]{meleard2012quasi} for more detail.
		%\footnote{Stefano: again, the reader has to pay attention that for this part of the theorem we  do not require the finite mean of $\mu$. Do we need to cover this case? If yes, I would suggest to put the last claim  into a remark containing also the example with $f_a(s):=1-(1-s)^a$.}
		%A family of QSDs can be constructed as follows: for $a\in(0,1)$, let $\vc\mu_a$ denote the initial distribution with generating function $f_a(z):=1-(1-z)^a$. It is clear that $E_{\mu_a}(Z_0)=\infty$. Then we can show that $G_a(z):=1-(1-G(z))^a$ is the generating function of the QLD of $\{Z_n\}$ starting with the  distribution $\vc\mu_a$. Since any QLD is a QSD, this shows the existence of uncountably many QSDs for the subcritical GW process.
	\end{remark}

	%Remark that the radius of convergence of $G(s)$ is greater than or equal to 1.

	In the remainder of the paper we assume that the GW process is subcritical ($m<1$) and
	%, making a slight abuse of language, 
	we use the term ``the quasi-stationary distribution of the GW process'' when referring to its Yaglom limit $\vc g$. We refer the reader to \cite[Section I. 8]{athreya2004branching} and \cite{meleard2012quasi} for more details about quasi-stationary distributions of GW processes. 
	
	%A \emph{linear fractional} Galton-Watson process is a process in which the offspring distribution is modified geometric with parameters $p_0$ and $p$, that is, for $j\geq 1$, $p_j=(1-p_0)(1-p)p^{j-1},$. It is subcritical if and only if $p_0>p$.
	%
	%\begin{lem}In the linear fractional case, the Yaglom limit is geometric with parameter $p/p_0$. 
	%\end{lem}
	
	It is worth mentioning another characterization for the quasi-stationary distribution of a GW process.
	Let $Q$ denote the truncated probability transition matrix of the process $\{Z_n\}$ corresponding to the (transient) positive integer states. Then the quasi-stationary distribution satisfies
	\begin{equation}\label{PF}\vc g\, Q=m\, \vc g,\end{equation}where $\vc g$ corresponds to the normalized Perron-Frobenius left eigenvector associated with the Perron-Frobenius eigenvalue $m$. The solution of \eqref{PF} is unique up to a multiplicative constant. This characterization is not specific to the GW process and holds for any absorbing Markov chain. Note that in the case of the GW process, the matrix $Q$ is dense and semi-infinite, hence solving \eqref{PF} is a non trivial task. A classical approach to approximate the solution of \eqref{PF} consists in considering the sequence of $N\times N$ northwest corner truncations of $Q$ and computing their respective left Perron-Frobenius eigenvectors \cite[Chapter 6]{Seneta}. In our context, this would require to solve  eigenvector problems for non sparse and possibly large (as $N$ grows) matrices. In this paper we introduce a new procedure that exploits more naturally the properties of the GW process. 
	
	%\subsection{Benchmark example: linear fractional branching processes}\label{sec:lin-frac}
	
	We end this section by defining the \emph{linear fractional branching processes}, which form a special class of GW  processes amenable to explicit computation. In these processes, the offspring distribution is modified geometric, that is, 
	$$p_j=(1-p_0)(1-p)p^{j-1},\quad j\geq 1,$$fully characterized by just two parameters: $p_0\in[0,1)$, and $p\in[0,1)$. Note here that $p_j\geq p_{j+1}$ for all $j\geq 1$. The mean offspring is given by
	$m={1-p_0}/{1-p},$ therefore the process is subcritical ($m<1$) if and only if $p_0>p$.
	The corresponding progeny generating function is given by
	$$P(z)=p_0+(1-p_0) \dfrac{(1-p) z}{1-pz}, \quad z\in[0,p^{-1}).$$
	
	It is not difficult to verify that the quasi-stationary distribution of a linear-fractional GW process is geometric with parameter $p/p_0$, that is, \begin{equation}\label{eq:lin-frac}g_j=\left( 1-\dfrac{p}{p_0}\right)\left(\dfrac{p}{p_0}\right)^{j-1},\; j\geq 1,\quad \textrm{and} \quad G(z)= \left(1 - \frac{p}{p_0}\right)\frac{z}{1-\frac{p}{p_0}z}.\end{equation}
	
	We shall use the linear fractional branching process in Section~\ref{sec:tests} as a benchmark tool to evaluate the quality of our numerical approximation methods for the computation of the quasi-stationary distribution.

	\section{Properties of $G(z)$}\label{sec:properties} 
	In this section we study the asymptotic behavior of the coefficients $g_j$, or in other words, the tail behavior of the quasi-stationary distribution $\vc g$.
	For computational purposes, we are interested in understanding the decay properties of these coefficients in order to ensure that a limited number of them is sufficient to describe $G(z)$ with arbitrary accuracy. For example, the existence of the $h$-th derivative $G^{(h)}(1)$ of $G(z)$ at $z=1$, which corresponds to the $h$-th factorial moment of $\vc g$,  provides an algebraic decay of (at least) order $h$, because $G^{(h)}(1) = \sum_{j \geq h} \frac{j!}{(j-h)!} g_j \approx \sum_{j \geq h} j^h g_j$.
	Exponential decay is directly linked to the radius of convergence of  $G(z)=\sum_{j\ge 0} g_jz^j$ and, consequently, to the domain of analyticity of $G(z)$. Indeed, given $R>0$, it is well known that a formal power series $\sum_{j\ge 0} g_jz^j$ defines an analytic function $G(z)$ on $\mathcal B(0,R)$ if and only if,  for all $ r\in(0,R)$ and $j\geq 0$, $|g_j|\leq\max_{|z|=r}|G(z)| \cdot r^{-j}$ \cite[Proposition IV.1]{flajo}. Since in our case the power series has real non-negative coefficients, $G(z)$ is analytic on $\mathcal B(0,R)$ if and only if
	\begin{equation}\label{eq:exp-decay}
	g_j\leq G(r) \cdot r^{-j}\qquad  \forall r\in(0,R), j\geq 0 .
	\end{equation}
	
	From a computational perspective, we would like $G(z)$  to be analytic on a disc with radius bigger than $1$. This would allow us to choose $r>1$ in $\eqref{eq:exp-decay}$, ensuring that at most $\left\lceil{\log(u^{-1}G(r) )}/{\log(r)}\right\rceil$ coefficients $g_j$ are above the machine precision $u$. This property is equivalent to having $G(z)$ analytic at $z=1$. The proof of the following proposition is provided in the appendix.
	\begin{proposition}\label{prop:1}
		Let $P(z)$ be analytic on $\mathcal B(0,r_P)$ with $r_P>1$. Then,
		$G(z)$ is analytic at $z=1$  if and only if  there exists $r_G>1$ such that $G(z)$ is analytic on $\mathcal B(0,r_G)$.
	\end{proposition} 
	In what follows we study the interplay between the regularity of the offspring distribution and that of the quasi-stationary distribution.
	\subsection{Derivatives of $G(z)$ at $z=1$}\label{sec:deriv}
	We start by looking at the existence of the derivatives of $G(z)$ at $z=1$. 
	The next theorem gives a necessary and sufficient condition on the offspring distribution for the mean of the quasi-stationary distribution to exist, that is, for $G^{(1)}(1)<\infty$.
	\begin{theorem}[Heathcote \textit{et al.} \cite{heathcote}]\label{thm:heath}
		\[G^{(1)}(1)<\infty\quad \Leftrightarrow \quad \sum_{j=2}^\infty j \log(j)\cdot p_j <\infty.\]
	\end{theorem}
	Higher order moments of the quasi-stationary distribution are studied in \cite{bagley}, where the author shows that $G^{(h)}(1)$ is finite if and only if $P^{(h)}(1)$ is finite for  $1<h\in\mathbb N$.  In the appendix we report a simpler and shorter proof of this fact, that relies on algebraic arguments only. 
	\subsection{Domain of analyticity of $G(z)$}
	Motivated by the preceding results, we wonder if assuming the analyticity of $P(z)$ on an open disc of radius bigger than $1$ is enough to ensure the same property for $G(z)$. The answer to this question is affirmative, as we formalize in the next theorem. 
	\begin{theorem}\label{thm:radius}
		If $P(z)$ has radius of convergence $r_P>1$, then $G(z)$ has radius of convergence $ r_G>1$.
	\end{theorem}
	
	In the appendix, we prove Theorem~\ref{thm:radius}  by directly showing  the  convergence of the Taylor series of the probability generating function $G(z)$ centred at $z=1$. This property can also be obtained by combining Proposition~\ref{prop:1}  with the \emph{linearization theorem} of Koenigs \cite{koenig}; see also \cite[Chapter II]{carleson}.
	
	The quantity $r_G$ is related to the basin of attraction of the attractor $1$ for the fixed point iteration $z_{k+1}=P(z_k)$ (see Section \ref{sec:prob}). 
	The latter in turn is  linked to the rightmost real solution of $z=P(z)$. 
	Observe that the offspring generating function $P(z)$ and all its derivatives are real and positive on the interval $[0,r_P)$, where $r_P$ denotes the radius of convergence of $P(z)$. In particular, the equation $z=P(z)$ can have either one or two solutions on the positive real semi-axis: $1$ and possibly $\widehat z>1$. We have that $z=P(z)$ only for $z=1$ in the case $P(z)\equiv 1-m + mz$ (the only degree $1$ polynomial that satisfies the assumptions on $P(z)$), and in the case $P(r_P)<r_P$ (see the right part of Figure~\ref{fig:P(z)}).
	Hence, we define 
	\begin{equation}\label{psiP}
	\psi_P:=\begin{cases}
	\infty & \text{if } P(z)\equiv 1-m + mz\\
	r_P& \text{if } P(r_P)<r_P\\
	\widehat z>1: \widehat z=P(\widehat z) &\text{otherwise,}
	\end{cases}
	\end{equation}
	so that $P(z)<z$ $\forall z\in(1,\psi_P)$.
	
	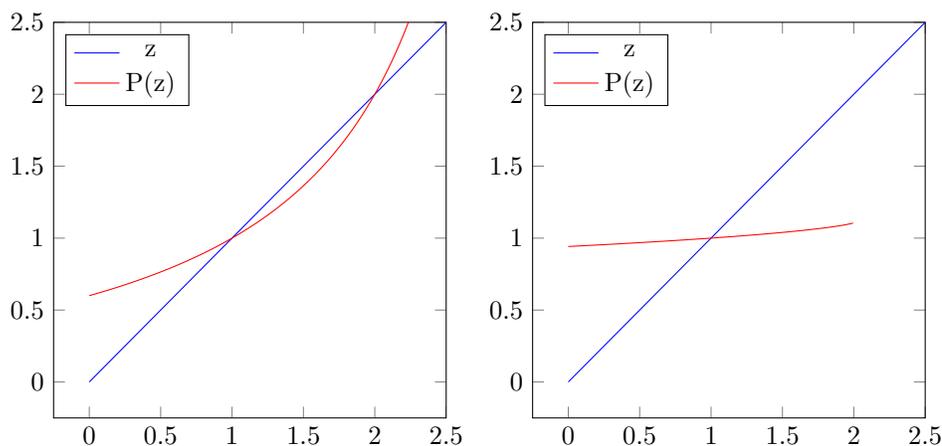
\begin{figure}
		\centering
		\begin{tikzpicture}
		\begin{axis}[width=.45\linewidth, height=.35\textheight,
		legend pos = north west, xmax=2.5,ymax=2.5]
		\addlegendimage{no markers,blue}
		\addlegendimage{no markers,red}
		\addplot table[x index=0, y index=1, mark = none]  {P_z_.dat};
		\addplot table[x index=0, y index=2, mark = none]  {P_z_.dat};
		
		\legend{z,P(z)};
		\end{axis}	
		\end{tikzpicture}
		\begin{tikzpicture}
		\begin{axis}[width=.45\linewidth, height=.35\textheight,
		legend pos = north west, xmax=2.5,ymax=2.5]
		\addlegendimage{no markers,blue}
		\addlegendimage{no markers,red}
		\addplot table[x index=0, y index=0, mark = none]  {tP_z_.dat};
		\addplot table[x index=0, y index=1, mark = none]  {tP_z_.dat};
		
		\legend{z,P(z)};
		\end{axis}	
		\end{tikzpicture}\caption{Intersections of $P(z)$ with the bisector of the first quadrant. On the left, $P(z):=0.6+0.4 \,(0.7 z)({1-0.3 z})^{-1}$. On the right, $P(z):=1+0.1\,(\widetilde P(z)-\widetilde P(1))$ with $\widetilde P(z):=\sum_{j\geq 1}\frac{z^j}{2^jj^2}$. In both cases $\psi_P =2$.}\label{fig:P(z)}		
	\end{figure}
	
	\begin{corollary}\label{cor:radius}
		Assume that  $P(z)$  has radius of convergence $r_P>1$. Then the following statements hold:
		\begin{itemize}
			\item[$(i)$] $ r_G\geq\psi_P$ where $\psi_P$ is defined in \eqref{psiP},
			\item[$(ii)$] $G(P(z)) = m\cdot G(z)+1-m,$ $\forall z\in \mathcal D(0,r_G)$.
		\end{itemize}
	\end{corollary}
	\begin{proof}
		By Theorem~\ref{thm:radius} we have $r_G>1$. To show $(i)$ we assume by contradiction that $1<r_G<\psi_P$, and consider the set $P^{-1}([1,r_G)):=\{z\in\mathbb R:\ P(z)\in[1,r_G)\}$; $P(z)< z$ on $(1,\psi_P)$ implies that $P^{-1}([1,r_G))=[1, y)$ with $y>r_G$. Rewriting \eqref{eq:G},  we can set $G(z)=(G(P(z))-1+m)/m$ for every $z\in[r_G,y)$, extending analytically the function on $[0,y)$. Since $|P(z)|\leq P(|z|)<|z|$ for $1<|z|<y$,  $G(z)$ can be extended analytically on the disc of radius $y$, leading to a contradiction. 
		
		The claim in $(ii)$ follows by continuity.
	\end{proof}
	\begin{remark}\label{rem:decay}
		The results in this section guarantee that we always have $r_P\geq r_G\geq \psi_P>1$. In particular, this provides the upper bound $\mathcal O(\psi_P^{-j})$ for the asymptotic behavior of $g_j$.
	\end{remark}	
	
	%\subsection{Benchmark example: linear fractional branching processes}\label{sec:lin-frac}
	%The so-called \emph{linear fractional branching processes} form a special class of GW branching processes amenable to explicit computation. In these processes, the offspring distribution is modified geometric, that is, 
	%$$p_j=(1-p_0)(1-p)p^{j-1},\quad j\geq 1,$$ and is fully characterized by just two parameters: $p_0\in[0,1)$, and $p\in[0,1)$. Note here that $p_j\geq p_{j+1}$ for all $j\geq 1$. The mean offspring is given by
	%$$m=\dfrac{1-p_0}{1-p},$$therefore the process is subcritical ($m<1$) if and only if $p_0>p$.
	%The corresponding progeny generating function is given by
	%$$P(z)=p_0+(1-p_0) \dfrac{(1-p) z}{1-pz}, \quad z\in[0,p^{-1}).$$
	%
	%It is not difficult to verify that the quasi-stationary distribution of a linear-fractional Galton-Watson process is geometric with parameter $p/p_0$, that is, \begin{equation}\label{eq:lin-frac}g_j=\left( 1-\dfrac{p}{p_0}\right)\left(\dfrac{p}{p_0}\right)^{j-1},\quad j\geq 1\quad \Longrightarrow \quad G(z)= \left(1 - \frac{p}{p_0}\right)\frac{z}{1-\frac{p}{p_0}z}.\end{equation}
	\begin{remark}Observe that, in the case of the linear fractional branching process, $r_P=\dfrac{1}{p}>r_G=\dfrac{p_0}{p}>1,$ and that, since $P(p_0/p)={p_0}/{p}$, $r_G=\psi_P$; this is in accordance with Corollary~\ref{cor:radius}.
	\end{remark}
	
	%We shall use the linear fractional branching process  in Section~\ref{sec:tests} as a benchmark tool to evaluate the quality of our numerical approximation methods for the computation of the quasi-stationary distribution.
	
	%according to Corollary~\ref{cor:radius}, $r_G={p_0}/{p}>1$, which verifies $P\left({p_0}/{p}\right)={p_0}/{p}$.  

\section{Methods for computing $G(z)$}\label{sec:alg}
In this section, we first review a method known in the literature to compute the quasi-stationary distribution of a general transient Markov chain; we then discuss another natural approach, based on probabilistic arguments, which suffers from some numerical drawbacks; finally we present our new algorithm.
\subsection{The returned process approach}\label{sec:returnmap}

This approach can be used to evaluate the quasi-stationary distribution $\vc g$ of a transient Markov chain $\{X_n\}_{n\geq 0}$ on $\mathbb{N}$, with the absorbing state $0$ assumed to be reached in finite time with probability one, regardless the initial state.  It relies on the idea that $\vc g$ can be approximated by the \emph{stationary distribution} $\vc\pi^\mu$ of a positive recurrent \emph{returned process} $\{X_n^\mu\}$, which is a Markov chain on the state space $\mathbb{N}_0$ that evolves exactly like the original process $\{X_n\}$ except at the times at which state 0 is visited, when it is instantly returned to a random positive state, chosen
according to a probability distribution $\vc\mu=(\mu_j)$ on $\mathbb{N}_0$. More precisely, if $T$ and $T^\mu$ denote the probability transition matrices of $\{X_n\}$ and $\{X_n^\mu\}$ respectively, then $T^\mu$ is defined as $T^\mu_{ij}:=T_{ij}+T_{i0}\,\mu_j$ for all $i,j\geq 1$, and $\vc g\approx \vc\pi^\mu$ where $\vc\pi^\mu$ solves $\vc\pi^\mu \,T^\mu=\vc\pi^\mu$, $\vc\pi^\mu\,\vc 1=1$;   for more detail, see for instance \cite{barbour2012total,blanchet2013empirical,van2013quasi}. 

The function $\vc\mu\rightarrow\vc\pi^\mu$ is contractive, and the quasi-stationary distribution  $\vc g$ satisfies $\vc g=\vc\pi^g$. Therefore, in practice, instead of sampling from a fixed distribution $\vc\mu$ every time the process visits state 0, we sample the return state from the empirical distribution of the returned process up to that time; after a large enough time, the empirical distribution of the returned process will then be a good approximation of the quasi-stationary distribution.
In summary, the returned process approach is a simulation-based method that works as follows:
\begin{itemize}\item[\emph{(i)}] start the Markov chain $\{X_n\}$ in a non-absorbing state $i\geq 1$;
	\item[\emph{(ii)}] simulate a sequence of states visited by $\{X_n\}$ normally as long as it does not reach state 0;
	\item[\emph{(iii)}] if the Markov chain hits state 0, do not record that transition and re-sample a new state $j\geq1$ at random according to
	the empirical estimate of the quasi-stationary distribution up
	until that time (that is, according to the proportion of time the chain  has spent in each transient state since the start of the simulation), then go back to step \emph{(ii)};
	%	In other words, we sample a new transient state with a probability proportional to the amount of time that such a state
	%	has been visited so far since the start of the simulation.
	\item[\emph{(iv)}] after a large enough time, the samples of states will be drawn approximately
	from the quasi-stationary distribution $\vc g$.
\end{itemize}

In our setting where $\{X_n\}$ corresponds to a GW  process, the simulation of $\{X_n\}$ requires the offspring of each individual to be simulated at each generation, which can be computationally demanding. In addition, a large number of generations generally need to be simulated in order to obtain a satisfactory approximation. This method is illustrated in Section~\ref{sec:tests}.

\subsection{A probabilistic interpolation approach}\label{sec:prob}
Here we discuss another method for computing $G(z)$ that has a probabilistic inspiration. This technique exploits Equation~\eqref{eq:G} in combination with the sequence $\{\widetilde z_k\}_{k\geq 0}$ recursively defined as
$$\widetilde z_{k+1}:=P(\widetilde z_k),$$  %\soph{if $m\leq 1$. This is a well-know result in the theory of branching processes [REF]}. %Moreover, the rate of convergence is linear and governed by $P^{(1)}(1)=m$:
%\begin{equation}\label{eq:rate}
%1-z_k\leq m^k (1-z_0).
%\end{equation}
with $\widetilde z_0=0$. This  leads to the recursion
$
G(\widetilde z_{k+1}) = m\cdot G(\widetilde z_k)+1-m,$ which, because $G(0)=0$, can be solved explicitly:
\begin{equation}\label{eq:Gzksol}G(\widetilde z_k)=1-m^k, \quad k\geq 0.\end{equation}
The sequence $\{\widetilde z_k\}_{k\in\mathbb N}$
has a probabilistic interpretation: $\widetilde z_k$ is the probability that the GW process becomes extinct by generation $k$, if it starts with a single individual in generation 0. 

In view of \eqref{eq:Gzksol}, $G(\widetilde z_k)$ also has a probabilistic meaning: $G(\widetilde z_k)=\sum_{j\geq 1}g_j\, \widetilde z_k^j$ is the probability that a subcritical GW process observed in its quasi-stationary regime dies within the next $k$ generations. So $m^k$ is the probability that, if we observe a subcritical GW process which has been living for a long time, it is still going to survive for at least $k$ generations. The mean offspring of a subcritical GW process can therefore be interpreted as the probability that the process survives one generation when it is in its quasi-stationary regime. A similar property is given in \mbox{\cite[Proposition 2]{meleard2012quasi}.}

From an algebraic perspective, we have at our disposal a sequence of nodes $\widetilde z_k\in[0,1]$ with $\widetilde z_0=0, \widetilde z_1=p_0, \widetilde z_2=P(p_0),\ldots$, for interpolating the function $G(z)$. However, there are two main issues: first, the set of nodes accumulates near the point $1$ and does not become dense in the interval $[0,1]$. Second, since we are interested in the coefficients $g_j$ we are forced to interpolate with respect to the monomial basis. This requires to solve linear systems (or linear least squares problem) with the Vandermonde matrix generated by the real nodes $\widetilde z_k$. It is well known that the latter is exponentially ill-conditioned with respect to the degree of the interpolant \cite{beck-cond,gautschi}. The performance of this approach is tested in Section~\ref{sec:tests}.
\subsection{A new algorithm}\label{sec:ev-int}
In this section we propose an algorithm for approximating the unknown quasi-stationary distribution $\vc g$ characterized by $G(z)$ when the offspring distribution $P(z)$ has a radius of convergence $r_P>1$. 

The method described in Section~\ref{sec:prob} suffers from the bad quality of the set of nodes used for interpolation. Here, we propose an alternative strategy that performs an approximate interpolation of $G(z)$ on the roots of unity, which is the most suited set for interpolating with respect to the monomial basis.  

We remark that the original functional equation $G(P(z))=mG(z)+1-m$ admits an infinite number of solutions of the form 
$G(z)= 1 + t\cdot f(z),$ where $t\in\mathbb C$ is an arbitrary constant, and $f(z)$ satisfies \begin{equation}\label{eq:f}f(P(z))-m\cdot f(z)\equiv 0.\end{equation} Once $f(z)$ is known, $G(z)$ can be obtained by imposing the boundary condition $G(0)=0$. 
Our strategy for computing $f(z)$ consists in discretizing the operator $f\rightarrow f\circ P -mf$ and looking for an eigenvector associated with its smallest eigenvalue. 

Observe that,  given $r\in(1,\psi_P)$, $\forall z\in r\cdot \mathcal S^1$ we have $|P(z)|<r$.
Therefore, for $z\in r\cdot \mathcal S^1$ we can use the  Cauchy integral formula and rewrite \eqref{eq:f} as
\begin{equation}\label{eq:f2}
\frac 1{2\pi \mathbf i}\int_{r\cdot \mathcal S^1}f(t)(t-P(z))^{-1}dt - m\cdot f(z)=0.
\end{equation}
Then, we replace the left-hand side of \eqref{eq:f2} with its  approximation obtained via the trapezoidal rule, choosing the scaled $n$-th roots of unity as nodes for integration. Even though other quadrature rules, e.g. Gauss-Legendre,  have a higher order of convergence in general, for the particular case of the integration on a disk, the trapezoidal rule has exponential convergence as $n$ increases \cite{trefethen}. This yields the integration scheme :
\begin{equation}\label{eq:approx}
\frac 1{2\pi \mathbf i}\int_{r\cdot \mathcal S^1}f(t)(t-P(z))^{-1}dt-mf(z) \approx\sum_{j=1}^{n} f(r\xi_j)\cdot \frac{r\xi_j}{n}\cdot (r\xi_j- P(z))^{-1}-mf(z), 
\end{equation}
where $\xi_j:=\exp({2\pi j\mathbf i}/{n})$, $j=1,2,\dots,n$.
Evaluating the right-hand side of \eqref{eq:approx} in the scaled $n$-th roots of unity provides the system of $n$ equations in the $n$ unknowns $f(r\xi_j)$:
\begin{equation}\label{eq:lin-sys}\sum\limits_{h=1}^{n} f(r\xi_h)\cdot \frac{r\xi_h}{n}\cdot (r\xi_h- P(r\xi_j))^{-1}-m\cdot f(r\xi_j) \approx 0\quad j=1,\dots, n.
\end{equation}
Rewriting \eqref{eq:lin-sys} in matrix form leads to the smallest-eigenpair problem:
\begin{equation}\label{eq:G3}
A \mathbf{v_f}= \lambda_{min}\mathbf{v_f},\qquad A=(a_{jh})_{j,h=1,\dots,n},\qquad a_{jh}= 
\begin{cases}
\frac{r\xi_h}{n(r\xi_h-P(r\xi_j))} &h\neq j\\
\frac{r\xi_h}{n(r\xi_h-P(r\xi_h))}-m &h=j.
\end{cases}
\end{equation}
Indeed, when $\lambda_{min}$ is the eigenvalue of smallest modulus of the matrix $A$, the vector $\mathbf{v_f}$ contains  approximations of the quantities  $\widetilde f(\xi_j):= f(r\cdot \xi_j)$, $j=1,\dots,n$, for a function $f$ that verifies \eqref{eq:f}.
Then, applying the \emph{Inverse Fast Fourier Transform} (IFFT) to  $\mathbf{v_f}$ provides the vector containing the (approximate) coefficients of the interpolating polynomial $\sum_{j=0}^{n-1}\widetilde f_j z^j$ for $\widetilde f(z)$ at the nodes $\xi_h$, for $h=1,\dots,n$.  In order to retrieve the (approximate) interpolating polynomial for $f(z)$ we rescale  the coefficients with the rule $f_j\gets \widetilde f_j/r^j$. Finally, we impose the boundary condition  $0=G(0)=1+tf(0)=1+tf_0$, which implies $t=-{1}/{f_0}$. This yields the following (approximate) interpolating polynomial $\widehat G(z)$ for $G(z)$:
\[
\widehat G(z):=\sum_{j=1}^{n-1}\widehat g_jz^j,\qquad \widehat g_j= -\frac{f_j}{f_0}.
\]
The procedure is summarized in Algorithm~\ref{alg:ev-int}; \textsc{Eigs}($A$) denotes any numerical method for computing the eigenvector of $A$ associated with its smallest eigenvalue. 

The construction of $A$ and \textsc{Eigs}($A$) constitute the bottlenecks of the algorithm. In particular, memorizing the full matrix and running \textsc{Eigs}($A$) in dense arithmetic ---  for example using the MATLAB command  \texttt{eigs(A, 1, 'SM')} --- provides a quadratic cost in storage and a cubic time consumption, respectively. In this setting, we have to consider $n$ less than $10^4$ in order to carry on the computations on a standard laptop. In Section~\ref{sec:rank}, we will show that it is possible to exploit the structure of the matrix $A$ for achieving a cheaper storage and an efficient implementation of \textsc{Eigs}, allowing us to consider higher values for $n$.
\begin{algorithm}
	\caption{Evaluation-Interpolation}\label{alg:ev-int}
	\begin{algorithmic}[1]
		\Procedure{Compute\_G}{$P(z), n, r$}\Comment{$r>P(r)>1$}
		\State $m \gets P^{(1)}(1)$
		\State $\mathbf{\xi}\gets\left(r\cdot e^{\frac{2\pi \mathbf i j}{n}}\right)_{j=1,\dots, n}$
		\State $A\gets\left(\frac{\xi_h}{n(\xi_h-P(\xi_j))}\right)_{j,h=1,\dots, n}$
		\State $A\gets A-m\cdot I_n$
		\State $\mathbf v_f\gets$ \Call{Eigs}{$A$}
		\State $\mathbf f\gets$ \Call{IFFT}{$\mathbf v_f$},\quad$\mathbf f\gets \left( \frac{f_j}{r^{j}}\right)_{j=0,\dots,n-1}$\Comment{$\sum_{j=0}^{n-1}f_jz^{j}$ interpolates $\widetilde f(z)$}
		\State $ \mathbf{\widehat g}\gets -\frac{1}{f_0}\mathbf f$,\quad $\widehat g_0\gets0$ 
		\State \Return $ \mathbf{\widehat g}$
		\EndProcedure
	\end{algorithmic}
\end{algorithm}
\subsection{Numerical tests}\label{sec:tests}
All the experiments reported in this paper have been performed on a Laptop with the dual-core Intel Core i7-7500U 2.70 GHz CPU, 256KB of level 2 cache, and 16 GB of RAM. The
algorithms are implemented in MATLAB and tested under MATLAB2017a, 
with MKL BLAS version 11.2.3 utilizing both cores.
\begin{example}
	As a first example, we consider  the linear-fractional GW process  with parameters $p_0=0.6$ and $p=0.3$. In this case, we know that the quasi-stationary distribution is given by $g_j=2^{-(j+1)}$, $j\geq 1$, see \eqref{eq:lin-frac}. This example allows us to evaluate the quality of the approximations resulting from the three algorithms that we have introduced. 
	
	In Figures~\ref{fig:graph} and \ref{fig:fit} we denote with the label ``Return map'' the procedure described in Section~\ref{sec:returnmap}, which we ran for $10^6$ generations.  
	With ``Interpolation'' we indicate the probabilistic interpolation approach of Section~\ref{sec:prob}, that generates the data set $\{(\widetilde z_k, 1-m^k)\}_{k=0,\dots,199}$ and  computes the corresponding fitting polynomial of degree $12$ using the \texttt{polyfit} function of MATLAB. Clearly, this yields estimates only for $g_j$ with $j=1,\dots,12$; however,  experimentally  we  notice that using a higher degree for the fitting polynomial provides noisy results. Moreover, adding points to the data set brings no benefits due to the convergence of the sequence $\widetilde z_k$ to its limit point. Finally, we run Algorithm~\ref{alg:ev-int} using $n=512$ integration nodes. 
	
	In Figure~\ref{fig:graph}-left we plot the approximated solutions $\widehat G(z)$ returned by the three methods over the unit interval. Since the three graphs are indistinguishable on this scale, in Figure~\ref{fig:graph}-right we zoom over the interval $[0.02,0.03]$ where we finally observe some differences. 
	In Figure~\ref{fig:fit}-left we report the approximated coefficients $\widehat g_j$ returned by the three methods and the true $g_j$'s. We let the index $j$ vary in the range $[0, 52]$ because, for $j\geq 53$, $g_j$ is below the machine precision. It is evident that  the outcome of the interpolation method strongly differs from the ones of the other algorithms and from the true solution; in addition the coefficients of the fitting polynomial are sometimes negative. In Figure~\ref{fig:fit}-right we report the relative errors $|(g_j-\widehat g_j)/g_j|$ of the three approaches; the results indicate  that the accuracy of Algorithm~\ref{alg:ev-int} largely exceeds that of the others, as it returns relative errors of magnitudes close to the  machine precision already for $512$ integration nodes. We also remark that the return map method with $10^6$ generations does not manage to provide non-zero estimates for the coefficients $g_j$ with $j>20$. Finally, we mention that the execution time of Return map was about $40$ seconds while Algorithm~\ref{alg:ev-int} and Interpolation needed less than $0.1$ seconds.
\end{example}

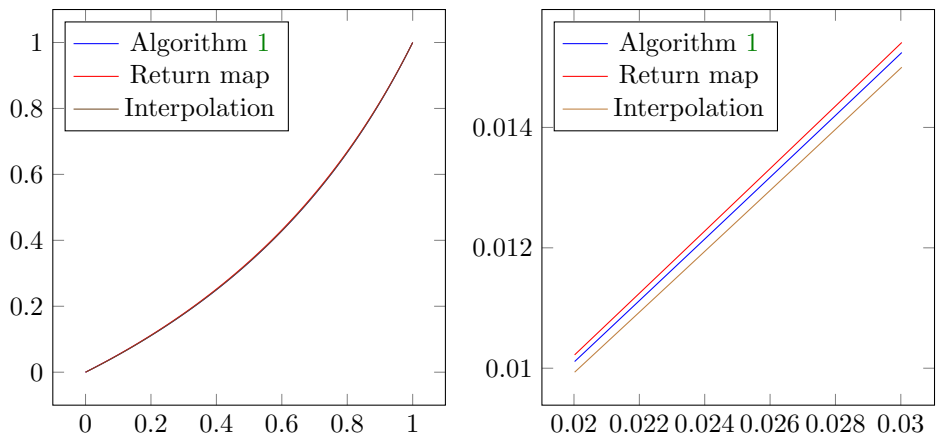
\begin{figure}\label{fig:graph}
	\begin{tikzpicture}
	\begin{axis}[width=.45\linewidth, height=.35\textheight,
	legend pos = north west, ymax=1.1, ymin=-0.1]
	%\addlegendimage{blue}
	%\addlegendimage{no markers,red}
	%\addlegendimage{no markers,brown}
	%\addlegendimage{no markers,black}
	%\addplot table[x index=0, y index=1,  mark = none]  {prob_zoom.dat};   
	\addplot+[solid, no markers] table[x index=0, y index=1, mark = none]  {comparison_graph.dat};
	\addplot+[solid, no markers] table[x index=0, y index=2, mark = none]  {comparison_graph.dat};
	\addplot+[solid, no markers] table[x index=0, y index=3, mark = none]  {comparison_graph.dat};
	%\addplot[domain=0:1,black] {(1-x/2)^(-1)-1 };
	\legend{Algorithm~\ref{alg:ev-int},Return map, Interpolation};
	\end{axis}
	\end{tikzpicture}~
	\begin{tikzpicture}
	\begin{axis}[width=.45\linewidth, height=.35\textheight,
	legend pos = north west, scaled y ticks = false, scaled x ticks = false, yticklabel style={
		/pgf/number format/precision=3,
		/pgf/number format/fixed}, xticklabel style={
		/pgf/number format/precision=3,
		/pgf/number format/fixed}]
	%\addlegendimage{no markers,blue}
	%\addlegendimage{no markers,red}
	%\addlegendimage{no markers,brown}
	\addplot+[select coords between index={20}{30}, color=blue, solid, no markers] table[x index=0, y index=1] {comparison_graph.dat};
	\addplot+[select coords between index={20}{30}, color=red, solid, no markers] table[x index=0, y index=2]{comparison_graph.dat};
	\addplot+[select coords between index={20}{30}, color=brown, solid, no markers] table[x index=0, y index=3] {comparison_graph.dat};
	%\addplot[domain=0.02:0.03,black] {(1-x/2)^(-1)-1 };
	\legend{Algorithm~\ref{alg:ev-int},Return map, Interpolation};
	\end{axis}	
	\end{tikzpicture}\caption{On the left; plots of the approximated solutions $\widehat G(z)$ returned by the three methods for the linear fraction branching process with $p_0=0.6$ and $p=0.3$. On the right; zoom of the picture on the left for $z\in[0.02,0.03]$.}
\end{figure}

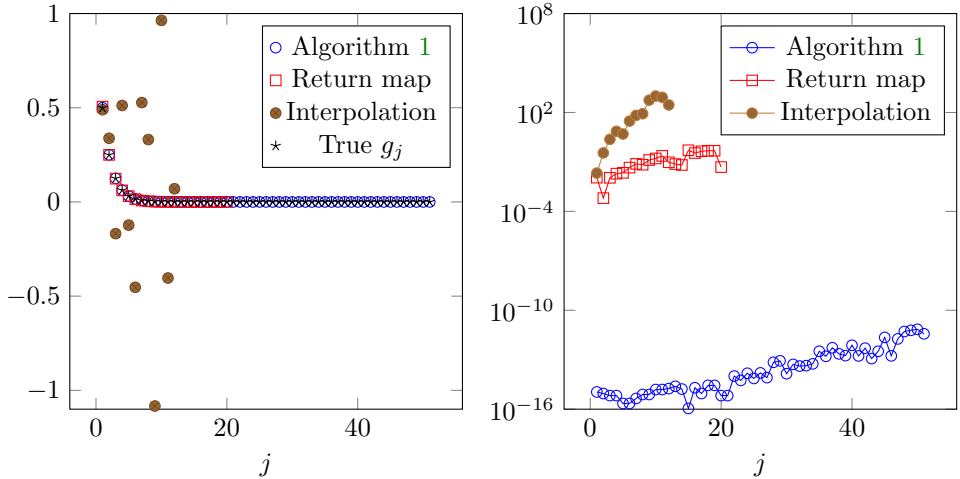
\begin{figure}\label{fig:fit}
	\begin{tikzpicture}
	\begin{axis}[width=.45\linewidth, height=.35\textheight,
	legend pos = north east, ymax=1, ymin=-1.1,xlabel = $j$]
	%\addlegendimage{blue}
	%\addlegendimage{no markers,red}
	%\addlegendimage{no markers,brown}
	%\addlegendimage{no markers,black}
	%\addplot table[x index=0, y index=1,  mark = none]  {prob_zoom.dat};   
	\addplot+[only marks, mark = o] table[x index=0, y index=1, mark = none]  {comparison_coef.dat};
	\addplot+[select coords between index={0}{19}, only marks, mark = square] table[x index=0, y index=2, mark = none]  {comparison_coef.dat};
	\addplot+[select coords between index={0}{11}, only marks] table[x index=0, y index=3, mark = none]  {comparison_coef.dat};
	\addplot+[only marks] table[x index=0, y index=4]  {comparison_coef.dat};
	\legend{Algorithm~\ref{alg:ev-int},Return map, Interpolation,True $g_j$};
	\end{axis}
	\end{tikzpicture}~
	\begin{tikzpicture}
	\begin{semilogyaxis}[width=.45\linewidth, height=.35\textheight,
	legend pos = north east,ymin=1e-16,ymax=1e8,xlabel = $j$]
	%\addlegendimage{no markers,blue}
	%\addlegendimage{no markers,red}
	%\addlegendimage{no markers,brown}
	\addplot+[mark = o] table[x index=0, y index=1] {comparison_err.dat};
	\addplot+[select coords between index={0}{19}, color=red, solid, mark = square] table[x index=0, y index=2]{comparison_err.dat};
	\addplot+[select coords between index={0}{11}, color=brown, solid, mark = *] table[x index=0, y index=3] {comparison_err.dat};
	\legend{Algorithm~\ref{alg:ev-int},Return map, Interpolation};
	\end{semilogyaxis}	
	\end{tikzpicture}\caption{On the left: approximated coefficients $\widehat g_j$ computed by the three methods and true coefficients $g_j$ of the linear fractional branching process with $p_0=0.6$ and $p=0.3$. On the right: relative errors $\left|\frac{g_j-\widehat g_j}{g_j}\right|$ of the three approaches.}
\end{figure}

\begin{example}\label{ex:1}
	We test Algorithm~\ref{alg:ev-int} on a randomly generated offspring distribution. More specifically, we set $P(z)$ equal to the polynomial of degree $8$ with the following coefficients: $p_0=0.838$, $p_1 = 0.008$, $p_2=0.031$, $p_3=0.011$, $p_4=0.021$, $p_5=0.029$, $p_6= 0.019$, $p_7=0.014$ and $p_8=0.029$. Consequently, we have $m=0.776$ and $\psi_P\approx 1.101$. The latter is estimated numerically as the rightmost solution of $z=P(z)$. The parameter $r$ is set equal to $\arg\min_{x\geq 1}P(x)-x$, which is obtained via the \texttt{fminsearch} function of MATLAB.
	This is because  we want to keep the magnitude of the quantities  $(P(r\xi_j)-r\xi_j)^{-1}$ (that are involved in the definition of the matrix $A$) under control.
	
	In Figure~\ref{fig:e2} (top-left), we show the performances of Algorithm~\ref{alg:ev-int} and the features of the computed solution. In particular, we report the residual defined as \[\text{Res}:=\max_{j=1,\dots,n}|\widehat G(P(\xi_j))-m\widehat G(\xi_j)-1+m|,\]
	and the sum of the coefficients $\widehat g_j$. We note that, in all our tests, the $\widehat g_j$'s are real and non-negative up to machine precision. As $n$ increases, the execution times scale cubically; the residual decreases rapidly to $0$ and the sum of the $\widehat g_j$'s converges to $1$. In Figure~\ref{fig:e2} (top-right), we plot the first $300$ coefficients $\widehat g_j$, computed in the case $n=8192$, and we compare their distribution with the decay rate $\psi_P^{-j}$, suggested by Corollary~\ref{cor:radius}. The outcome confirms the sharpness of the decay rate. 
\end{example}
\begin{example}\label{ex:2}
	We consider a test analogous to the one in Example~\ref{ex:1}, but with a mean offspring $m$ closer to $1$. We set the coefficients of $P(z)$ as $p_0=0.782$, $p_1 = 0.016$, $p_2=0.045$, $p_3=0.038$, $p_4=0.037$, $p_5=0.008$, $p_6= 0.009$, $p_7=0.04$ and $p_8=0.025$, which   yields $m=0.942$ and $\psi_P\approx 1.026$. Apart from the case $n=8192$, we note the presence of negative coefficients $\widehat g_j$ whose order of magnitude ranges from $10^{-3}$ (for $n=256$) to $10^{-5}$ (for $n=4096$).  The results reported in the bottom of Figure~\ref{fig:e2} highlight the fact that we are still far from convergence; indeed the residual is much higher  than in Example~\ref{ex:1}, and the sum of the $\widehat g_j$'s is not close to $1$. Finally, the slope of the decay of the coefficients is further from the theoretical estimate.
\end{example}

Example~\ref{ex:2} suggests that for $m$ close to $1$, higher values of $n$ are needed in order to reach a satisfactory accuracy. This leads us to deal with a large scale matrix $A$. In the next sections we propose an improvement of Algorithm~\ref{alg:ev-int} for treating this case.
\begin{figure}
	\centering
	\begin{minipage}{0.45\textwidth}
		\begin{filecontents*}{e1.dat}
256	0.047386	0.086675	-1.8165e-09	0.71826
512	0.10788		0.013963		0	0.96031
1024	0.43893		0.00049624	-2.5656e-27	0.99864
2048	1.6177		8.8515e-07	-3.231e-36	1
4096	9.922		3.8032e-12	-8.3516e-36	1
8192	78.207		2.0381e-15	-2.1666e-35	1
\end{filecontents*}
\pgfplotstabletypeset[%
		sci zerofill,
		columns={0,1,2,4},
		columns/0/.style={column name=$n$},
		columns/1/.style={column name=Time (s), fixed},
		columns/2/.style={column name=Res},
		%columns/3/.style={column name=$\min g_j$},
		columns/4/.style={column name=$\sum \widehat g_j$},
		]{e1.dat}
	\end{minipage}~\begin{minipage}{0.52\textwidth}
		\bigskip 
		
		\begin{tikzpicture}
		\begin{semilogyaxis}[width=.95\linewidth, height=.27\textheight,
		legend pos = north east,
		xlabel = $j$]
		\addplot+[mark=o] table[x index=0, y index=1]  {e1G.dat};   
		\addplot[domain=0:305,dashed] {1.101^(-x) };
		\legend{$\widehat g_j$,$\psi_P^{-j}$};
		\end{semilogyaxis}
		\end{tikzpicture}
	\end{minipage}
	%\caption{Example~\ref{ex:1}. On the left, performances of Algorithm~\ref{alg:ev-int} as  $n$ increases. On the right, comparison between the estimated coefficicients of $G(z)$, in the case $n=8192$, and the decay suggested by Corollary~\ref{cor:radius}.}
	%\label{fig:e1}
	%\end{figure}
	%\begin{figure}
	\centering
	\begin{minipage}{0.45\textwidth}
		\begin{filecontents*}{e2.dat}
256	0.057566	0.099403	-0.003462	0.73287
512	0.10763		0.07529		-0.003105	0.74327
1024	0.39404		0.043347	-0.0013237	0.9741
2048	1.5472		0.021419	-0.00024173	0.66786
4096	9.9373		0.024882	-4.0046e-05	0.65197
8192	79.883		0.013284	-3.2369e-42	0.85669
\end{filecontents*}
\pgfplotstabletypeset[%
		sci zerofill,
		columns={0,1,2,4},
		columns/0/.style={column name=$n$},
		columns/1/.style={column name=Time (s), fixed},
		columns/2/.style={column name=Res},
		%columns/3/.style={column name=$\min g_j$},
		columns/4/.style={column name=$\sum \widehat g_j$},
		]{e2.dat}
	\end{minipage}~\begin{minipage}{0.52\textwidth}
		\bigskip 
		
		\begin{tikzpicture}
		\begin{semilogyaxis}[width=.95\linewidth, height=.27\textheight,
		legend pos = north east,
		xlabel = $j$, ymax=1e3]
		\addplot+[mark=o] table[x index=0, y index=1]  {e2G.dat};   
		\addplot[domain=0:305,dashed] {1.026^(-x) };
		\legend{$\widehat g_j$,$\psi_P^{-j}$};
		\end{semilogyaxis}
		\end{tikzpicture}
	\end{minipage}
	\caption{Example~\ref{ex:1} (top) and Example~\ref{ex:2} (bottom). On the left, performances of Algorithm~\ref{alg:ev-int} as  $n$ increases. On the right, comparison between the estimated coefficicients of $G(z)$, in the case $n=8192$, and the decay suggested by Corollary~\ref{cor:radius}.}
	\label{fig:e2}
\end{figure}
\subsection{Rank structure in the matrix $A$}\label{sec:rank}
We now take a closer look at the matrix $A$ in \eqref{eq:G3}. This matrix can be written as
\[
A = C_{P,r}^{(n)} \cdot\diag\left(\frac{r\xi_1}{n},\dots,\frac{r\xi_n}{n} \right)- mI_n,
\]
where the $n\times n$ matrix $C_{P,r}^{(n)}=(c_{hj})$ is defined as $c_{hj}:=(r\xi_h-P(r\xi_j))^{-1}$ with $\xi_h=\exp({2\pi\mathbf i h}/{n})$. 
The aim of this subsection is to show that the matrix $A$  is well-approximated by a multiple of the identity matrix plus a low-rank matrix. In view of the Eckart-Young theorem \cite[Chapter 3]{horn91}, this is equivalent to showing that the singular values $\sigma_k$ of $C_{P,r}^{(n)}\cdot \diag\left(\frac{r\xi_1}{n},\dots,\frac{r\xi_n}{n} \right)$ rapidly become negligible, with respect to $\sigma_1$, as $k$ increases. 
We note that the analysis is reduced to studying the $n\times n$ matrix $C_{P,r}^{(n)}$, because the multiplication by the diagonal matrix  does not alter the ratio $\sigma_k/\sigma_1$.

The matrix $C_{P,r}^{(n)}$ belongs to a well-studied class of structured matrices that we introduce within the next definition.

\begin{definition}
	A matrix  $(c_{hj})\in\mathbb C^{m\times n}$ is called a \emph{Cauchy matrix} if there exist two vectors $\mathbf x\in\mathbb C^{m}$ and $\mathbf y\in\mathbb C^{n}$ such that $c_{hj}=(x_h-y_j)^{-1}$. We call $\mathbf x,\mathbf y$ the \emph{generators} and we  denote the matrix $(c_{hj})$ by  $C(\mathbf x,\mathbf y)$.
\end{definition} 

The behavior of the singular values of $C(\mathbf x,\mathbf y)$ can be linked to the configurations of the two sets $\mathcal X:=\{x_h\}_{h=1,\dots,m}$ and $\mathcal Y:=\{y_j\}_{j=1,\dots, n}$
%\footnote{can't we just talk about the configuration of the two vectors $\mathbf x,\mathbf y$ instead of sets?}  
in the complex plane. There are many results in the literature on the singular value decay of Cauchy matrices whose corresponding sets  $\mathcal X$ and $\mathcal Y$ are separated in some sense \cite{pan14,rokhlin,chandra}. 

\subsubsection{Bounds linked to polynomial approximation}
The existence of accurate low-rank approximations of $C(\mathbf x,\mathbf y)$ is implied by the existence of low-degree \emph{separable approximations} $\widetilde a(x,y)=\sum_{j=1}^kg_j(x)h_j(y)$ of the function $a(x,y):={(x-y)^{-1}}$, over the set $\mathcal X\times \mathcal Y$, see \cite[Section 4]{hack}. A possible way to determine a separable approximation of $a(x,y)$ consists in considering its truncated Taylor expansion with respect to one of the two variables. Intuitively, for using the Taylor expansion we need the set $\mathcal X\times \mathcal Y$ to be well separated from the singularities of $a(x,y)$. This is encoded in the following definition.

\begin{definition}\label{def:sep}
	Given $\theta \in (0,1)$ and $c\in\mathbb C$ we say that two sets $\mathcal X,\mathcal Y\subset \mathbb C$ are \emph{$(\theta, c)$-separated} if for every $x\in\mathcal X$ and $y\in\mathcal Y$ we have ${|y-c|}\leq \theta {|x-c|}$.	
\end{definition}
The property in Definition~\ref{def:sep} provides an explicit exponential decay in the singular values of $C(\mathbf x,\mathbf y)$, as stated in the next result. 
\begin{theorem}[Chandrasekaran \textit{et al.} \cite{chandra}, Section 2.2]\label{thm:decay1}
	Let $\{x_h\}_{h=1,\dots,m},\{y_j\}_{j=1,\dots,n}\subset \mathbb C$ be $(\theta,c)$-separated for a certain $\theta\in(0,1)$ and a complex center $c$. Then, $\forall k\in\mathbb N$ there exist $U\in\mathbb C^{m\times k}$ and $V\in\mathbb C^{n\times k}$ such that
	\[
	\norm{C(\mathbf x,\mathbf y)-UV^*}_2\leq \frac{\theta^k}{(1-\theta)\delta}\sqrt{mn},
	\]
	where $\delta :=\min_{h=1,\dots,m}|c-x_h|$.
\end{theorem}
\begin{remark}\label{rem:sing}
	By the Eckart-Young theorem, Theorem~\ref{thm:decay1} provides the bound
	\[
	\sigma_{k+1}\leq\frac{\theta^k}{(1-\theta)\delta}\sqrt{mn},\qquad k=1,2,\dots,
	\]
	for the singular values of $C(\mathbf x,\mathbf y)$.
\end{remark}	

In order to apply Theorem~\ref{thm:decay1} in our framework, we need to understand better where the function $P(r\cdot z)$ maps the unit circle.
Relying on the stochasticity properties of the coefficients $p_j$, $\forall z\in \mathcal S^1$ we have
\begin{equation}\label{eq:P(z)}
|P(r\cdot z)-p_0|= \left|\sum_{j=1}^\infty p_j(rz)^j\right|\leq \sum_{j=1}^\infty p_jr^j=P(r)-p_0,
\end{equation}
i.e., $P(r\cdot \mathcal S^1)$ can be enclosed into the circle $p_0+\alpha(r)\mathcal S^1 $, with $\alpha(r):=P(r)-p_0$. This is at the basis of the next lemma.
\begin{lemma}\label{lem:decay1}
	Let $n\in\mathbb N$, $P(z)=\sum_{j\geq 0}p_jz^j$, $p_j\ge 0\ \forall j\ge 0$, $P(1)=1$, $P^{(1)}(1)\in (0,1)$, and let $r>1$ be such that $r>P(r)$. Then $\forall k= 1,\dots, n-1$ we have
	\begin{equation}\label{eq:bound1}
	\sigma_{k+1}(C_{P,r}^{(n)})\leq \frac{\theta^kn}{(1-\theta)(r-p_0)},
	\end{equation}
	where $\theta := \frac{P(r)-p_0}{r-p_0}\in(0,1)$. 
\end{lemma}

\begin{proof}
	Let us consider $\mathbf x=\left(r\xi_j\right)_{j=1,\dots,n}$ and $\mathbf y=\left(P(r\xi_j)\right)_{j=1,\dots,n}$, so that $C_{P,r}^{(n)}=C(\mathbf x, \mathbf y)$. In light of \eqref{eq:P(z)}, we have
	\[
	\frac{|y_j-p_0|}{|x_h-p_0|}=\frac{|P(r\xi_j)-p_0|}{|r\xi_j-p_0|}\leq \frac{P(r)-p_0}{r-p_0},
	\]
	that is $\mathcal X$ and $\mathcal Y$ are $\left(\frac{P(r)-p_0}{r-p_0}, p_0\right)$-separated. Then, the claim follows by applying Theorem~\ref{thm:decay1} and Remark~\ref{rem:sing} to $C(\mathbf x, \mathbf y)$.
\end{proof}
\begin{remark}\label{rem:sing2}
	The decay rate $\theta$ in Lemma~\ref{lem:decay1} depends on the parameter $r$. We note that the strategy of minimizing the difference $P(r)-r$ on the interval $(1,\psi_P)$ (when we choose $r$) also minimizes the quantity $\theta$. 	
\end{remark}
\begin{example}\label{ex:bound}
	We proceed to test the quality of the bound \eqref{eq:bound1} by considering  the linear fractional family of progeny distributions with $p=1/2$:
	\[
	P(z)= p_0+(1-p_0)\sum_{j\geq 1}\left(\frac z2\right)^j=p_0 + \frac{z(1-p_0)}{2-z},\qquad p_0\in(0,1).
	\]	
	For these distributions, we have $m=P^{(1)}(1)\in(0,1)$ if and only if $p_0\in\left(1/2, 1\right)$. In particular, when $p_0$ is close to $1/2$,  $m$ is close to $1$ and the interval $(1,\psi_P)$ 	shrinks drastically. This yields a ratio $(P(r)-p_0)/(r-p_0)$ close to $1$ and consequently a slow decay. The opposite behavior is obtained when $p_0$ tends to $1$. 
	
	In the left panels of Figures~\ref{fig:bound2}  we report the singular values of the matrix $C_{P,r}^{(n)}$ together with the bound  \eqref{eq:bound1}, for $n=1000$, in the cases  $p_0=0.55$ and $p_0=0.95$. In the right panels of these figures we plot the three curves: $r\mathcal S^1$, $P(r\mathcal S^1)$ and $\alpha(r)\mathcal S^1+p_0$. The provided bound is quite informative  in the case $p=0.95$ but it is useless when $p=0.55$. Indeed, whenever the distance between $\mathcal X$  and $\mathcal Y$ tends to $0$, e.g. in the case $p_0=0.55$, the Taylor expansion of $(x-y)^{-1}$ converges slowly and this translates in slowly decaying bounds for the singular values of $C(\mathbf x,\mathbf y)$. 
\end{example}

\begin{figure}
	\centering
	\begin{minipage}{0.45\textwidth}
		\begin{tikzpicture}
		\begin{semilogyaxis}[width=.95\linewidth, height=.32\textheight,
		legend style={font=\tiny},
		legend pos = south west,
		xlabel = $k$, xmax = 100]
		\addplot table[x index=0, y index=1]  {bound1.dat};   
		\addplot+[color=red, solid, mark = line] table[ x index=0, y index=2] {bound1.dat};
		\legend{$\sigma_k(C_{P,r}^{(n)})$,Bound \eqref{eq:bound1}};
		\end{semilogyaxis}
		\end{tikzpicture}
	\end{minipage}~\begin{minipage}{0.45\textwidth}
		%		\bigskip 		
		\begin{tikzpicture}
		\begin{axis}[width=.95\linewidth, height=.32\textheight,
		legend style={font=\tiny},
		legend pos = north east,
		xlabel = Real axis, ylabel = Imaginary axis, xmin=-2, xmax=2, ymin=-1.5, ymax = 2.5]
		\addplot+[color=red, solid, mark = line] table[x index=3, y index=4]  {bound1.dat};   
		\addplot+[color=blue, solid, mark = line] table[x index=5, y index=6]  {bound1.dat}; 
		\addplot+[color=cyan, solid, mark = line] table[x index=7, y index=8]  {bound1.dat}; 
		\legend{$r\mathcal S^1$, $P(r\mathcal S^1)$, $ p_0+\alpha(r)\mathcal S^1$};
		\end{axis}
		\end{tikzpicture}
	\end{minipage}
	%	\caption{Case $p_0=0.55$. On the left, the first singular values of the matrix $C_{P,r}^{(n)}$, with $n=1000$, compared with the bound in \eqref{eq:bound1}. On the right, the regions $r\mathcal S^1$ (in red) and $P(r\mathcal S^1)$ (in blue) containing the sets $\mathcal X$ and $\mathcal Y$, respectively. In light blue the curve $p_0+\alpha(r)\mathcal S^1$ that encloses $\mathcal Y$, by \eqref{eq:P(z)}}
	%	\label{fig:bound1}
	%	\end{figure}
	%	
	%	\begin{figure}
	\centering
	\begin{minipage}{0.45\textwidth}
		\begin{tikzpicture}
		\begin{semilogyaxis}[width=.95\linewidth, height=.32\textheight,
		legend style={font=\tiny},
		legend pos = north east,
		xlabel = $k$, xmax = 50]
		\addplot [select coords between index={0}{17}, blue, mark = *]  table[x index=0, y index=1]  {bound2.dat};   
		\addplot+[color=red, solid, mark = line] table[ x index=0, y index=2] {bound2.dat};
		\legend{$\sigma_k(C_{P,r}^{(n)})$,Bound \eqref{eq:bound1}};
		\end{semilogyaxis}
		\end{tikzpicture}
	\end{minipage}~\begin{minipage}{0.45\textwidth}
		%	\bigskip 			
		\begin{tikzpicture}
		\begin{axis}[width=.95\linewidth, height=.32\textheight,
		legend style={font=\tiny},
		legend pos = north east,
		xlabel = Real axis, ylabel = Imaginary axis, xmin=-2.7, xmax=2.7, ymin=-2, ymax=3.4]
		\addplot+[color=red, solid, mark = line] table[x index=3, y index=4]  {bound2.dat};   
		\addplot+[color=blue, solid, mark = line] table[x index=5, y index=6]  {bound2.dat}; 
		\addplot+[color=cyan, solid, mark = line] table[x index=7, y index=8]  {bound2.dat}; 
		\legend{$r\mathcal S^1$, $P(r\mathcal S^1)$, $p_0+\alpha(r)\mathcal S^1$};
		\end{axis}
		\end{tikzpicture}
	\end{minipage}
	\caption{Case $p_0=0.55$ (top) and case $p_0=0.95$ (bottom). On the left, the first singular values of the matrix $C_{P,r}^{(n)}$, with $n=1000$, compared with the bound in \eqref{eq:bound1}. On the right, the regions $r\mathcal S^1$ (in red) and $P(r\mathcal S^1)$ (in blue) containing the sets $\mathcal X$ and $\mathcal Y$, respectively. In light blue the curve $p_0+\alpha(r)\mathcal S^1$ that encloses $\mathcal Y$, by \eqref{eq:P(z)}}
	\label{fig:bound2}
\end{figure}
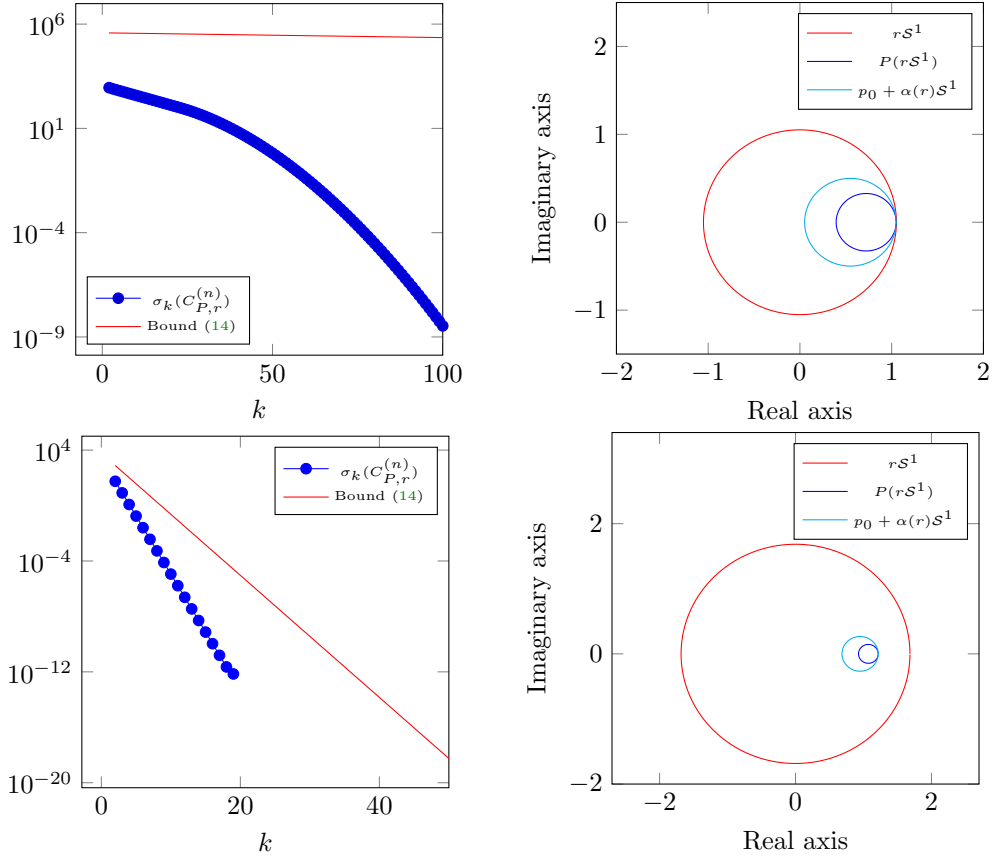

\subsubsection{Bounds linked to rational approximations}
A link between the quantities $\sigma_k(C_{P,r}^{(n)})$ and certain rational approximation problems has been described in \cite{beck17}. This motivates the presence of a fast decay in a wider class of configurations for $\mathcal X$ and $\mathcal Y$.
\begin{theorem}[Beckermann and Townsend \cite{beck17}, Section 4]\label{thm:zol}
	Let $\mathcal R_{k,k}$ denote the set of rational functions
	of the form $r(z)={p(z)}/{q(z)}$, where $p(z)$ and $q(z)$ are polynomials of degree at most
	$k$. Then
	\begin{equation}\label{eq:zol}
	\frac{\sigma_{k+1}(C(\mathbf x,\mathbf y))}{\norm{C(\mathbf x,\mathbf y)}_2}\leq Z_k(\mathcal X,\mathcal Y):=\min_{r\in\mathcal R_{k,k}}\frac{\max_{\mathcal X}|r(z)|}{\min_{\mathcal Y}|r(z)|}.
	\end{equation}
\end{theorem} 
Relation \eqref{eq:zol} bounds the relative singular values with $Z_k(\mathcal X,\mathcal Y)$, usually called the
$k$-th \emph{Zolotarev number}.  Intuitively, these quantities become small when $\mathcal X$ and $\mathcal Y$ are well separated. For example, if $\mathcal X=\{|z|\geq r_1\}$ and $\mathcal Y=\{|z|\leq r_2\}$ with $r_1>r_2$ then \begin{equation}\label{eq:annulus}Z_k(\mathcal X,\mathcal Y)\leq \frac{\max_{\mathcal X}|z^{-k}|}{\min_{\mathcal Y}|z^{-k}|}=\left(\frac{r_2}{r_1}\right)^k.\end{equation} 
It can be shown that the inequality in \eqref{eq:annulus} can be replaced by an equality, see  \cite[Theorem 2.1]{starke}.
%Equality in \eqref{eq:annulus} can be shown to hold, see  \cite[Theorem 2.1]{starke} for the reverse inequality.
In light of \eqref{eq:annulus}, the bound in Theorem~\ref{thm:decay1} can be retrieved from Theorem~\ref{thm:zol} by including all the points $P(z_j)$ in a disc centered at $0$.

In addition, Zolotarev numbers enjoy the following properties.
\begin{proposition}[Akhiezer \cite{akhiezer}]\label{prop:zol}
	Let $\mathcal X,\mathcal Y$ be disjoint subsets of $\mathbb C$ and assume $Z_k(\mathcal X,\mathcal Y)$ is defined as in \eqref{eq:zol}. Then the following properties hold:
	\begin{itemize}
		\item[(i)] Let $\mathcal W,\mathcal Z\subset\mathbb C$ and assume $\mathcal X\subseteq\mathcal W$ and $\mathcal Y\subseteq \mathcal Z$. Then $Z_{k}(\mathcal X,\mathcal Y)\leq Z_{k}(\mathcal W,\mathcal Z)$, $\forall k\in\mathbb N$.
		\item[(ii)] Let $T(z)$ be any M\"obius transform, then $Z_{k}(\mathcal X,\mathcal Y)= Z_{k}(T(\mathcal X),T(\mathcal Y))$, $\forall k\in\mathbb N$.
	\end{itemize}
\end{proposition}
For generic complex sets, it  appears  to  be  difficult  to  derive  explicit bounds  for $Z_k(\mathcal X,\mathcal Y)$. However, our situation can be re-casted to the case where $\mathcal X,\mathcal Y$ are the two connected components of the complement of an open annulus.
\begin{lemma}\label{lem:decay2}
	Under the assumptions of Lemma~\ref{lem:decay1}, we have
	\begin{equation}\label{eq:bound2}
	\frac{\sigma_{k+1}(C_{P,r}^{(n)})}{\norm{C_{P,r}^{(n)}}_2}\leq \theta^k \qquad \forall k\geq 0,
	\end{equation}
	where $\theta=\frac{(r-\beta)(P(r)-\alpha)}{(r-\alpha)(P(r)-\beta)}$, $\alpha= \frac{2 p_0 P(r) - P(r)^2 + r^2 + \sqrt{(2 p_0 P(r) - P(r)^2 + 
			r^2)^2-4 p_0^2 r^2}}{2p_0}$ and $\beta =\frac{r^2}{\alpha}$. 
\end{lemma}
\begin{proof}
	By Theorem~\ref{thm:zol} and Proposition~\ref{prop:zol} (i) we have  ${\sigma_{k+1}(C_{P,r}^{(n)})}/{\norm{C_{P,r}^{(n)}}_2}\leq Z_k(\mathcal W,\mathcal Z)$ where $\mathcal W=r\cdot \mathcal S^1$ and $\mathcal Z=\{|z-p_0|\leq P(r)-p_0\}$. The idea is to consider a M\"obius transformation that maps $\mathcal W$ and $\partial \mathcal Z$ into two concentric circles centred in the origin and that maps the inner part of $\mathcal Z$ into the inner part of the smaller disc. The  M\"obius transformation that satisfies these requirements is given by $T(z)={(z-\alpha)}/{(z-\beta)}$ where the coefficients $\alpha,\beta$ are common \emph{inverse points for the circles $\mathcal W$ and $\partial \mathcal Z$} \cite[Section 4.2]{Henrici}. Algebraically, $\alpha$ and $\beta$ solve the system
	\[
	\begin{cases}
	\alpha \beta =r^2\\
	(\alpha - p_0) (\beta - p_0) = (P(r) - p_0)^2,
	\end{cases}
	\]
	and $T(z)$ maps $\mathcal W$ into ${(r-\alpha)}/{(r-\beta)}\cdot \mathcal S^1$ and $\mathcal Z$ into $\mathcal D\left(0, {(P(r)-\alpha)}/{(P(r)-\beta)}\right)$. Hence, the claim follows by applying Proposition~\ref{prop:zol} (ii) with $T(z)$ and the bound in \eqref{eq:annulus}.
\end{proof}
\begin{figure}
	\centering
	\begin{minipage}{0.45\textwidth}
		\begin{tikzpicture}
		\begin{semilogyaxis}[width=.95\linewidth, height=.32\textheight,
		legend style={font=\tiny},
		legend pos = north east,
		xlabel = $k$, xmax = 50]
		\addplot+[select coords between index={0}{50}, color=blue, solid, mark = *] table[x index=0, y index=1]  {bound3.dat};   
		\addplot+[select coords between index={0}{50}, color=red, solid, mark = line] table[ x index=0, y index=2] {bound3.dat};
		\legend{$\frac{\sigma_k}{\sigma_1}\left(C_{P,r}^{(n)}\right)$,Bound \eqref{eq:bound2}};
		\end{semilogyaxis}
		\end{tikzpicture}
	\end{minipage}~\begin{minipage}{0.45\textwidth}
		%		\bigskip 		
		\begin{tikzpicture}
		\begin{axis}[width=.95\linewidth, height=.32\textheight,
		legend style={font=\tiny},
		legend pos = north east,
		xlabel = Real axis, ylabel = Imaginary axis, xmin=-1.5,xmax=1.5, ymin=-1, ymax = 2]
		\addplot+[color=red, solid, mark = line] table[x index=3, y index=4]  {bound3.dat};   
		\addplot+[color=blue, solid, mark = line] table[x index=5, y index=6]  {bound3.dat}; 
		\addplot+[color=cyan, solid, mark = line] table[x index=7, y index=8]  {bound3.dat}; 
		\legend{$T(r\mathcal S^1)$, $T(P(r\mathcal S^1))$, $ T(p_0+\alpha(r)\mathcal S^1)$};
		\end{axis}
		\end{tikzpicture}
	\end{minipage}
	\caption{Case $p_0=0.55$. On the left, the first relative singular values of $C_{P,r}^{(n)}$, with $n=1000$, compared with the bound in \eqref{eq:bound2}. On the right, the image of the regions $r\mathcal S^1$, $P(r\mathcal S^1)$and $p_0+\alpha(r)\mathcal S^1$ under the M\"oebius transformation $T(z)={(z-\alpha)}/{(z-\beta)}$}
	\label{fig:bound3}
\end{figure}
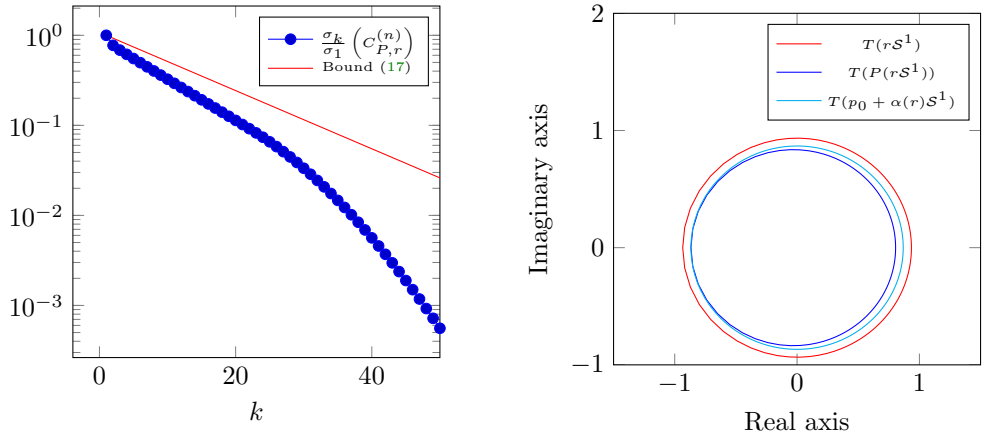

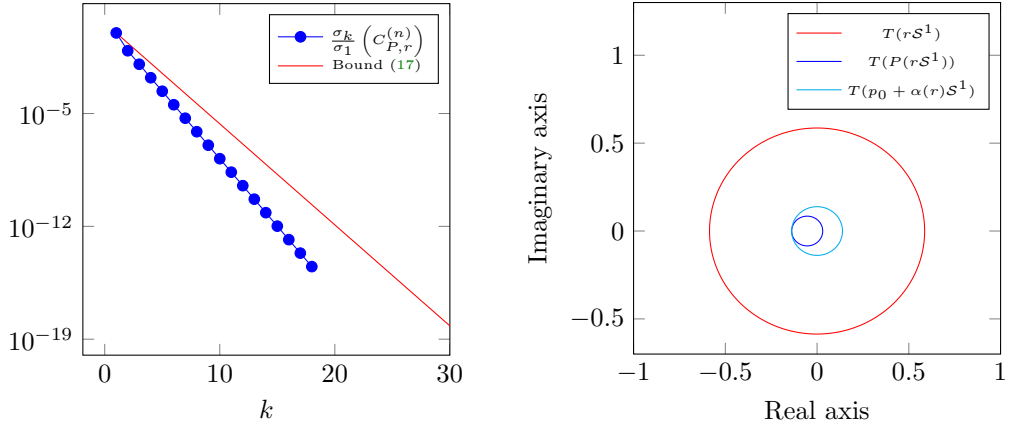
\begin{figure}
	\centering
	\begin{minipage}{0.45\textwidth}
		\begin{tikzpicture}
		\begin{semilogyaxis}[width=.95\linewidth, height=.32\textheight,
		legend style={font=\tiny},
		legend pos = north east,
		xlabel = $k$, xmax = 30]
		\addplot [select coords between index={0}{17}, blue, mark = *]  table[x index=0, y index=1]  {bound4.dat};   
		\addplot+[select coords between index={0}{30}, color=red, solid, mark = line] table[ x index=0, y index=2] {bound4.dat};
		\legend{$\frac{\sigma_k}{\sigma_1}\left(C_{P,r}^{(n)}\right)$,Bound \eqref{eq:bound2}};
		\end{semilogyaxis}
		\end{tikzpicture}
	\end{minipage}~\begin{minipage}{0.45\textwidth}
		%	\bigskip 			
		\begin{tikzpicture}
		\begin{axis}[width=.95\linewidth, height=.32\textheight,legend style={font=\tiny},
		legend pos = north east,
		xlabel = Real axis, ylabel = Imaginary axis,xmin=-1,xmax=1,ymin=-0.7, ymax=1.3]
		\addplot+[color=red, solid, mark = line] table[x index=3, y index=4]  {bound4.dat};   
		\addplot+[color=blue, solid, mark = line] table[x index=5, y index=6]  {bound4.dat}; 
		\addplot+[color=cyan, solid, mark = line] table[x index=7, y index=8]  {bound4.dat}; 
		\legend{$T(r\mathcal S^1)$, $T(P(r\mathcal S^1))$, $T(p_0+\alpha(r)\mathcal S^1)$};
		\end{axis}
		\end{tikzpicture}
	\end{minipage}
	\caption{Case $p_0=0.95$. On the left, the first relative singular values of $C_{P,r}^{(n)}$, with $n=1000$, compared with the bound in \eqref{eq:bound2}. On the right, the image of the regions $r\mathcal S^1$, $P(r\mathcal S^1)$and $p_0+\alpha(r)\mathcal S^1$ under the M\"oebius transformation $T(z)={(z-\alpha)}/{(z-\beta)}$}
	\label{fig:bound4}
\end{figure}
\begin{example}
	The qualitative behavior of the bound in Lemma~\ref{lem:decay2} on the case considered in Example~\ref{ex:bound} is shown in Figures~\ref{fig:bound3}--\ref{fig:bound4}. We also plot the action of the M\"oebius transform ${(z-\alpha)}/{(z-\beta)}$ on the sets $r\mathcal S^1$, $P(r\mathcal S^1)$ and $p_0+\alpha(r)\mathcal S^1$. Inequality \eqref{eq:bound2} achieves a sharper description of the slope of the decay, especially for what concerns the first singular values. We expect that a complete sharpness is not attainable due to the fact that we are using an estimate for $P(r\cdot \mathcal S^1)$. Moreover, in order to capture the superlinear behavior that appears in the case $p_0=0.55$ one might consider the Zolotarev numbers on the discrete sets $\mathcal X,\mathcal Y$, see for example \cite{beck-gry}. The latter is beyond the scope of this paper.  	
\end{example}		
\subsection{Exploiting the structure in Algorithm~\ref{alg:ev-int}}\label{sec:exploit}
The results in Section~\ref{sec:rank} ensure that the matrix $A$ in \eqref{eq:G3} can be well approximated by the sum $UV^*-mI_n$ where $U,V\in\mathbb C^{n\times k}$   are tall and skinny matrices (i.e., $k\ll n$) that constitute a low-rank approximation of $C_{P,r}^{(n)}\cdot \diag\left(\frac{r\xi_1}{n},\dots,\frac{r\xi_n}{n}\right)$. More specifically, we just need to store two $n\times k$ matrices  and one scalar in order to represent $A$, making the memory consumption linear with respect to the number of integration nodes. 

The factors $U$ and $V$ are computed by means of  the \emph{Adaptive Cross Approximation with partial pivoting} (ACA) \cite[Algorithm 1]{borm} whose pseudocode is reported in Algorithm~\ref{alg:aca}. This procedure is heuristic but experimentally effective for the case studies reported in this paper. Note that Algorithm~\ref{alg:aca} only needs to have a cheap access to the entries of its argument, so there is no need to form the full matrix $C_{P,r}^{(n)}$ in order to compress it. In all the numerical tests  that call Algorithm~\ref{alg:aca}, we have set $\tau=10^{-10}$.

In order to keep the rank $k$ of the approximation  as low as possible one might apply a  re-compression technique --- e.g. \cite[Algorithm 2.17]{hack} --- to the factors $U$ and $V$ returned by Algorithm~\ref{alg:aca}. Experimentally, we notice that this strategy does not bring any advantage in term of computational time, hence we do not apply any recompression method.

We also use the structure of $A$ in the computation of the eigenvector associated with its smallest eigenvalue. Indicating  with $\lambda_{min}(\cdot)$ the  eigenvalue of smallest magnitude of the (matrix) argument,  we notice that if $\lambda_{min}(UV^*-mI_n)\neq -m$, then
$
\lambda_{min}(UV^*-mI_n)=\lambda_{min}(V^*U-mI_k)=:\widetilde \lambda.
$
Moreover, if $\mathbf v_{min}$ is such that $(V^*U-mI_k)\mathbf v_{min}=\widetilde \lambda \mathbf v_{min}$ then 
$U\mathbf v_{min}$ satisfies
$
(UV^*-mI_n)U\mathbf v_{min}= \widetilde \lambda U\mathbf v_{min}.
$
This suggests the procedure outlined in Algorithm~\ref{alg:eigs-lr}, that has a $\mathcal O(nk^2 + k^3)$ cost.
%Moreover, we consider the inverse power method \cite[Section 7.6.1]{matrix-comp}  for computing the eigenvector associated to the smallest eigenvalue. Indeed, this only requires to compute the action of $A^{-1}$ on a vector and in our case this task can be performed efficiently with the use of the Sherman-Morrison-Woodbury identity \cite[Section 2.1.3]{matrix-comp}. Combining these tools  we get Algorithm~\ref{alg:inv-pow} that has a $\mathcal O(n)$ cost per iteration. As stopping criterion, we require the sine of the angle between two subsequent iterates --- $\norm{x_k-x_{k-1}x_{k-1}^*x_k}_2$ --- to be less than $10^{-8}$. 
Replacing \texttt{eigs} in Line 6 of Algorithm~\ref{alg:ev-int} with Algorithm~\ref{alg:eigs-lr}, we get the structured procedure for computing the coefficients of $G(z)$ that is summarized in Algorithm~\ref{alg:lr-ev-int}. 

The method is tested on the problematic Example~\ref{ex:2}, where we considered larger values of $n$. We observe that the computed coefficients $\widehat g_j$ are positive, up to machine precision, and they sum up to $1$ in all cases.  A complete picture of this test is shown in Figure~\ref{fig:e3}. The rank of the approximation of $C_{P,r}^{(n)}$ (returned by Algorithm~\ref{alg:aca}) is reported in the  column with the label ``rank''. This quantity  seems to stabilize around a value less than $500$. When the rank growth is limited (as in the last two numerical tests)  the computational times confirm the almost linear complexity with respect to $n$. The residual error remains around the value $10^{-10}$, this quantity being a reasonable estimate of the approximation error which affects the outcome of ACA. By decreasing the value of $\tau$ when running Algorithm~\ref{alg:aca}, the method provides more accurate results for $n=1.31\cdot 10^5, 2.62\cdot 10^5$.
\begin{algorithm}
	\caption{Adaptive cross approximation with partial pivoting}\label{alg:aca}
	\begin{algorithmic}[1]
		\Procedure{ACA}{$C,\tau$}\Comment{Computes the low-rank approximation $C\approx UV^*$}
		\State Choose a starting $i_1^*$ 
		\State Set $k\gets 1$, $U\gets[\ ]$, $V\gets [\ ]$ 
		\For {$k=1,2,\dots$}
		\State $v\gets C_{i_k^*,:}-U_{i_k^*,:}\ V^*$	
		\State $j_k^*\gets\arg\max_j |v_j|$
		\State $u\gets\left( C_{:,j_k^*}-U\ V_{j_k^*,:}^*\right) / v_{j_k^*}$
		\State $U\gets[U, u]$
		\State $V\gets[V, v^*]$
		\If{$\norm{u}_2\norm{v}_2<\tau$}
		\State break
		\EndIf
		\State $i_k^*\gets \arg\max_{i\neq i_k^*}|u_j|$
		\EndFor
		\State \Return $ U,V$
		\EndProcedure
	\end{algorithmic}
\end{algorithm}
%\begin{algorithm}
%	\caption{Low-rank Inverse Power method}\label{alg:inv-pow}
%	\begin{algorithmic}[1]
%		\Procedure{Inv\_Pow\_LR}{$U,V,m$}\Comment{Computes the smallest eigenv. of $UV^*-mI$}
%		\State $\mathbf{x_0} \gets$\texttt{rand(n)},\qquad $\mathbf{x_0}\gets\frac{\mathbf{x_0}}{\norm{\mathbf{x_0}}_2}$ \Comment{unit random complex vector}
%		\State $M\gets I-\frac{1}{m}V^*U$
%		\For {$k=1,2,\dots$}
%		\State $\mathbf{x_k}\gets -\frac{1}{m}\mathbf{x_{k-1}} - \frac{1}{m^2}UM^{-1}V^*\mathbf{x_{k-1}}$
%		\State $\mathbf{x_k}\gets\frac{\mathbf{x_k}}{\norm{\mathbf{x_k}}_2}$
%		\If{$\norm{\mathbf{x_k}-\mathbf{x_{k-1}}\mathbf{x_{k-1}}^*\mathbf{x_k}}_2<10^{-8}$}
%		\State break
%		\EndIf
%		\EndFor
%		\State \Return $ \mathbf{x_k}$
%		\EndProcedure
%	\end{algorithmic}
%\end{algorithm}
\begin{algorithm}
	\caption{}\label{alg:eigs-lr}
	\begin{algorithmic}[1]
		\Procedure{Eigs\_LR}{$U,V,m$}\Comment{Computes the smallest eigenvector of $UV^*-mI$}
		\State $v=\Call{Eigs}{V^*U-mI_k}$ 
		\State \Return $ Uv$
		\EndProcedure
	\end{algorithmic}
\end{algorithm}
\begin{algorithm}
	\caption{Low-rank Evaluation-Interpolation}\label{alg:lr-ev-int}
	\begin{algorithmic}[1]
		\Procedure{Compute\_G\_LR}{$P(z), n, r$}\Comment{$r>P(r)>1$}
		\State $m \gets P^{(1)}(1)$
		\State $\mathbf{\xi}\gets\left(r\cdot e^{\frac{2\pi \mathbf i j}{n}}\right)_{j=1,\dots, n}$
		\State $[U,V]\gets$\Call{ACA}{$C_{P,r}^{(n)}$}
		\State $V \gets \frac{1}{n}V\cdot \diag\left(\xi\right)$
		\State $\mathbf v_f\gets$ \Call{Eigs\_LR}{$U,V,m$}
		\State $\mathbf f\gets$\Call{IFFT}{$\mathbf v_f$},\quad$\mathbf f\gets \left( \frac{f_j}{r^{j}}\right)_{j=0,\dots,n-1}$
		\State $\mathbf{\widehat g}\gets -\frac{1}{f_0}\mathbf f$,\quad $\widehat g_0\gets0$ 
		\State \Return $\mathbf{\widehat g}$
		\EndProcedure
	\end{algorithmic}
\end{algorithm} 

\begin{figure}
	\centering
	\begin{minipage}{0.55\textwidth}
		\begin{filecontents*}{e3.dat}
16384		2.2791	0.00040953	-1.8475e-26		0.99622	264
32768		6.8755	5.7951e-07	-9.1198e-25		0.99999	366
65536		21.67	1.1534e-10	-3.6736e-25		1	465
1.3107e+05	49.889	4.3319e-10	-3.2967e-25		1	471
2.6214e+05	113.19	5.9423e-10	-1.3778e-24		1	475
\end{filecontents*}
\pgfplotstabletypeset[%
		sci zerofill,
		columns={0,1,2,4, 5},
		columns/0/.style={column name=$n$},
		columns/1/.style={column name=Time (s), fixed},
		columns/2/.style={column name=Res},
		%columns/3/.style={column name=$\min g_j$},
		columns/4/.style={column name=$\sum \widehat g_j$},
		columns/5/.style={column name=rank}
		]{e3.dat}
	\end{minipage}~\begin{minipage}{0.42\textwidth}
		\bigskip 
		
		\begin{tikzpicture}
		\begin{semilogyaxis}[width=.95\linewidth, height=.27\textheight,
		legend pos = north east,
		xlabel = $j$, ymax=1e3]
		\addplot+[mark=o] table[x index=0, y index=2]  {e3G.dat};   
		\addplot[domain=0:305,dashed] {1.026^(-x) };
		\legend{$\widehat g_j$,$\psi_P^{-j}$};
		\end{semilogyaxis}
		\end{tikzpicture}
	\end{minipage}
	\caption{Example~\ref{ex:2}. On the left, performances of Algorithm~\ref{alg:lr-ev-int} as  $n$ increases. On the right, comparison between the estimated coefficients of $G(z)$, in the case $n=262144$, and the decay suggested by Corollary~\ref{cor:radius}.}
	\label{fig:e3}
\end{figure}
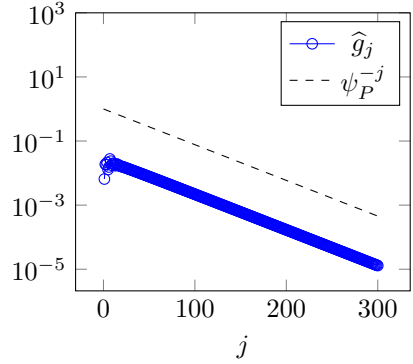  
\subsection{Taking advantage of self-similarity when $m\approx 1$}
The closer $m$ to $1$, the higher the rank $k$ of the approximation of  $C_{P,r}^{(n)}\cdot \diag\left(\frac{r\xi_1}{n},\dots,\frac{r\xi_n}{n}\right)$. Therefore, to limit resource consumption, when $m\approx 1$, one can think about exploiting self-similarity of Cauchy matrices. Indeed, every sub matrix of $C(\mathbf x,\mathbf y)$ is again a Cauchy matrix whose generators are subvectors $\mathbf{\tilde x}:=\left( x_j\right)_{j\in J_1},$ $\mathbf{\tilde y}:=\left( y_j\right)_{j\in J_2}$. In our setting, $\mathbf x,\mathbf y$ represent samplings of closed curves that rotate counterclockwise, so, intuitively, selecting  disjointed subsets $J_1,J_2$ of $\{1,\dots,n\}$ provides well separated sets of nodes. This translates in saying that the rank of the off-diagonal blocks is  smaller than the rank of $C(\mathbf x,\mathbf y)$ and sometimes the difference is substantial; see the example reported in Figure~\ref{fig:partitioning}. In these cases, it is advisable to rely on representations like $\mathcal H^2$ \cite{borm} and HSS \cite{chandra-hss} that aim at compressing the off-diagonal submatrices while keeping the small diagonal blocks in the dense format. Adopting this strategy still allows us to store and operate with matrices with a $\mathcal O(n)$ complexity. The $\mathcal H^2$ and HSS representations of the matrix $A$ can be obtained by applying the algorithms described in \cite{martinson, smash}. The use of these more sophisticated formats is beyond the scope of this paper and might be the subject of future investigations.

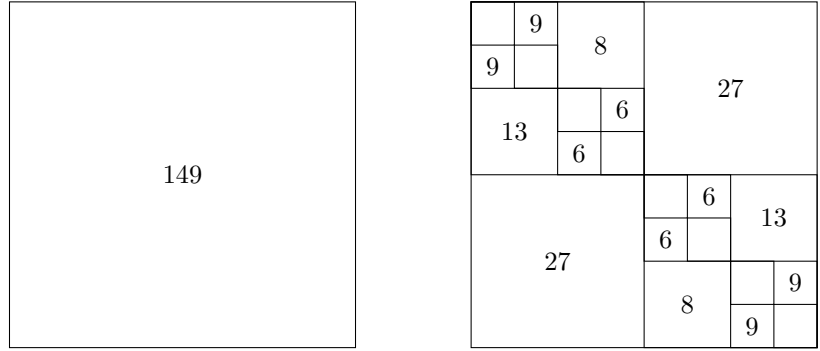
\begin{figure}
	
	\centering
	\begin{tikzpicture}[x = 1.15cm, y = 1.15cm]
	\node at (2, 2){$149$};
	\draw (0,0) rectangle (4,4);
	
	\end{tikzpicture}~~~~~~~~~~~~~\begin{tikzpicture}[x = 1.15cm, y = 1.15cm]
	\node at (1,1)     {$27$};
	\node at (3,3)     {$27$};
	\node at (2.5,.5)  {$8$};
	\node at (3.5,1.5) {$13$};
	\node at (.5,2.5)  {$13$};
	\node at (1.5,3.5) {$8$};
	\node at (.25, 3.25){$9$};
	\node at (.75, 3.75){$9$};
	\node at (1.25, 2.25){$6$};
	\node at (1.75, 2.75){$6$};
	\node at (2.25, 1.25){$6$};
	\node at (2.75, 1.75){$6$};
	\node at (3.25, .25){$9$};
	\node at (3.75, .75){$9$};
	\draw (0,0) rectangle (4,4);
	
	\foreach \i in {0, 2} {
		\draw (\i,4-\i) rectangle (\i+2,2-\i);
	}
	\foreach \i in {0,...,7} {
		\draw (\i*.5,4-\i*.5) rectangle (\i*.5+.5,3.5- \i*.5);
	}
	\foreach \i in {0, ..., 3} {
		\draw (\i,4-\i) rectangle (\i+1,3-\i);
	}	
	\end{tikzpicture}
	\caption{Rank (left) and off-diagonal rank distribution (right) of the matrix $C_{P,r}^{(n)}$. We set $n=4096$, $r\approx 1.0078$ and $P(z)$ equal to the polynomial of degree $7$ with coefficients $p_0=0.765$, $p_1 = 0.016$, $p_2=0.039$, $p_3=0.034$, $p_4=0.049$, $p_5=0.043$, $p_6= 0.005$ and $p_7=0.049$ which yields $m=0.98$.}
	\label{fig:partitioning}
\end{figure}
	\section{Multitype processes}\label{sec:multi}
	In a multitype GW process, individuals of each type can give birth to children of various types according to a progeny distribution specific to the parental type. For the sake of clarity, in this section we consider the two-type case; analogous arguments can be applied to the case of an arbitrary (yet finite) number of types. 
	
	For $j=1,2$ and $h,k\in \mathbb N$, we denote by $p^{(j)}_{h,k}$ the probability that a type-$j$ individual produces $h$ children of type $1$ and $k$ children of type $2$; we let $P_j(x,y):=\sum_{(h,k)\in\mathbb N^2}p^{(j)}_{h,k}\ x^h y^k$ denote the offspring generating function of a type-$j$ individual, and $\vc P(x,y):=(P_1(x,y),P_2(x,y))^\top$.
	The mean progeny matrix is defined as
	\begin{equation}\label{MM}M:=\left.
	\begin{bmatrix}
	\frac{\partial P_1(x,y)}{\partial x}&\frac{\partial P_1(x,y)}{\partial y}\\
	\frac{\partial P_2(x,y)}{\partial x}&\frac{\partial P_2(x,y)}{\partial y}
	\end{bmatrix}\right|_{(x,y)=(1,1)},
	\end{equation}and is assumed throughout the section to be positive regular, that is, $\exists n\geq1$ such that $(M^n)_{ij}>0$ for all $i,j=1,2$. Analogue to the mean offspring $m$ in the single-type case, the Perron-Frobenius eigenvalue $\rho$ of $M$ determines the criticality of the process. Here we consider the subcritical case $\rho<1$ with almost sure extinction regardless the type of the initial individual \cite[Theorem 7.1]{harris}.
	
	Let $\vc Z_n:=(Z_{n,1},Z_{n,2})$ denote the population size in both types at generation $n\geq 0$. We assume that the second moments of the offspring distribution are finite, that is, $\mathbb E(Z_{n,i}Z_{n,j}\,|\,\vc Z_0=\vc e_k)<\infty$ for all $i,j,k=1,2$. If the process starts with a single individual of type $j$, then the conditional probability distribution of $\vc Z_n$, given $\vc Z_n\neq 0$, converges as $n\rightarrow\infty$ to a limiting bivariate quasi-stationary distribution $\vc g^{(j)}:=(g^{(j)}_{h,k})_{(h,k)\in\mathbb N^2}$, whose probability generating function $$G_j(x,y):=\sum\limits_{(h,k)\in\mathbb N^2}g^{(j)}_{h,k}\ x^h y^k, \quad x,y\in[0,1],$$ satisfies
	\begin{equation}\label{2Deq}G_j(\vc P(x,y)) = \rho \cdot G_j(x,y) +
	1-\rho,\quad j=1,2;\end{equation}see for instance \cite[Theorem 9.1]{harris} and \cite[Chapter 4, Theorem 2]{athreya2004branching} for a stronger result without the second moment assumption. Here we aim at computing $G_j(x,y)$
	with $g^{(j)}_{h,k}\geq 0$ and $G_j(0,0)=0$ for $j=1,2$.
	Note that $G_1(x,y)= G_2(x,y)$, therefore we remove the  subscript of $G(x,y)$ and the superscript of the coefficients $g_{h,k}$.

	\subsection{The bivariate linear fractional case}\label{sec:2d-lin-frac}
	
	In the two-type case, the progeny distributions of a linear fractional GW process take the form
	$$
	P_1(x,y)=\dfrac{s_{11} x + s_{12} y +b_1}{c_1 x+c_2 y +d},\qquad
	P_2(x,y)=\dfrac{s_{21} x + s_{22} y +b_2}{c_1 x+c_2 y +d}$$ for some real parameters $s_{ij}$, $b_i$, $c_i$, $d$, $i=1,2$. Defining the matrix $S:=(s_{ij})_{i,j=1,2}$ and the vector $\vc c:=(c_1,c_2)$, the mean progeny matrix is given by $M=(S-\vc 1\otimes \vc c)/(c_1+c_2+d)$. Let $\vc\nu$ denote its left Perron-Frobenius eigenvector, normalised such that $\nu_1+\nu_2=1$. Finally, let $\vc t:=(t_1,t_2)=-\vc c/d$, $t_0=1-(t_1+t_2)$, $\vc w:=(w_1,w_2)=\vc t/t_0$, and $\vc\mu:=\vc w (I-M)^{-1}/(1+\vc w (I-M)^{-1}\vc 1)$. Then, the generating function of the quasi-stationary distribution is given by 
	$$G(x,y)=\dfrac{(\nu_1-\mu_1) x+(\nu_2-\mu_2) y}{1-\mu_1 x-\mu_2 y},$$
	see \cite[Theorem 1]{joffe2006multitype}.
	This provides the explicit expression
	\begin{equation}\label{eq:ghk}
	g_{h,k}= (\nu_1-\mu_1)\binom{h+k-1}{k}\mu_1^{h-1}\mu_2^k+(\nu_2-\mu_2)\binom{h+k-1}{h}\mu_1^{h}\mu_2^{k-1},
	\end{equation}
	where the binomial coefficient $\binom{a}{b}$ is assumed to be equal to $0$ whenever $b>a$.

	\subsection{Extension of Algorithm~\ref{alg:ev-int} to two dimensions}
	%For simplicity we describe the extension of Algorithm~\ref{alg:ev-int} to the \soph{two-type} case. Analogous \soph{arguments} can be applied to the case of an arbitrary (yet finite) number  of types.  
	%
	%
	%We consider the \soph{two-dimensional} version of \eqref{eq:G},
	%\begin{equation}\label{eq:G2D}
	%G(P(x,y)) = \rho \cdot G(x,y) + \begin{bmatrix}
	%1-\rho\\
	%1-\rho
	%\end{bmatrix},\qquad P(x,y)=\begin{bmatrix}
	%P_1(x,y)\\P_2(x,y)
	%\end{bmatrix},
	%\end{equation}
	%where, for $j=1,2$,  $P_j(x,y)=\sum\limits_{(h,k)\in\mathbb N^2}p^{(j)}_{h,k}\ x^h y^k$ \soph{is the generating function of the offspring distribution of a type-$j$ individual, that is, $p^{(j)}_{h,k}\geq 0$ denotes the probability of generating $h$ children of type $1$ and $k$ children of type $2$,} $P_j(1,1)=1$, and $\rho < 1$ is the spectral radius of the mean progeny matrix
	%\soph{\[M:=\left.
	%\begin{bmatrix}
	%\frac{\partial P_1(x,y)}{\partial x}&\frac{\partial P_1(x,y)}{\partial y}\\
	%\frac{\partial P_2(x,y)}{\partial x}&\frac{\partial P_2(x,y)}{\partial y}
	%\end{bmatrix}\right|_{(x,y)=(1,1)}.
	%\]}
	%We aim at computing
	%\[
	%G(x,y)=\begin{bmatrix}
	%G_1(x,y)\\G_2(x,y)
	%\end{bmatrix},\qquad  G_j(x,y)=\sum\limits_{(h,k)\in\mathbb N^2}g^{(j)}_{h,k}\ x^h y^k
	%\]
	%that verifies $g^{(j)}_{h,k}\geq 0$ and $G_j(0,0)=0$ for $j=1,2$.
	%Note that \soph{$G_1\equiv G_2$, therefore we remove the superscript of the coefficients $g_{h,k}$}. 
	
	Similar to the one-dimensional case, we define for $j=1,2$
	\[
	\psi_{P_j}=\begin{cases}
	\infty&\text{if $P_j(x,x) $ is of degree $1$}\\
	r_{P_j}&\text{if $P_j(r_{P_j},r_{P_j})<r_{P_j} $}\\
	\widehat z \in(1,\infty):\ \widehat z=P_j(\widehat z,\widehat z)
	& \text{otherwise}, \end{cases}
	\]
	where $r_{P_j}$ denotes the radius of convergence of $P_j(x,x)$.
	In particular, given $r_1\in(1,\psi_{P_1})$ and $r_2\in(1,\psi_{P_2})$, the function $P(x,y)$ is holomorphic on a open neighborhood of the polydisc $\mathcal D(0,r_1)\times \mathcal D(0,r_2)$. 
	
	As in  the previous case, 
	every function of the form $G(x,y)= 1 + t\cdot f(x,y)$ such that $t\in\mathbb C$ and\begin{equation}\label{eq:eig-2D}
	f(\vc P(x,y))-\rho\cdot f(x,y)\equiv 0,
	\end{equation}
	solves \eqref{2Deq}. By construction, 
	$\forall (x,y)\in (r_1 \mathcal S^1\times r_2 \mathcal S^1 )$ we have $|P_j(x,y)|< r_j$, $j=1,2$, hence
	applying the multivariate Cauchy integral formula to \eqref{eq:eig-2D} provides
	\begin{equation}\label{eq:G2D2}
	\frac{1}{(2\pi\vc i)^2}\int_{r_1\mathcal S^1}\int_{r_2\mathcal S^1}\frac{f(\tilde x,\tilde y)}{(\tilde x-P_1(x,y))(\tilde y-P_2(x,y))}\ d\tilde x \ d\tilde y-\rho \ f(x,y)=0.
	\end{equation}
	Then, we approximate \eqref{eq:G2D2} by means of the composite trapezoidal rule, i.e., for both integrals we select as nodes of integration  the scaled $n$-th roots of unity: 
	\begin{equation}\label{eq:2D-trap}
	\frac{1}{(2\pi\vc i)^2}\int_{r_1\mathcal S^1}\int_{r_2\mathcal S^1}\frac{f(\tilde x,\tilde y)}{(\tilde x-P_1(x,y))(\tilde y-P_2(x,y))}\ d\tilde x \ d\tilde y\approx \sum_{h,k=0}^{n-1} \frac{f(r_1\xi_h, r_2\xi_k)\cdot r_1r_2\xi_{h+k}}{{n^2} (r_1\xi_h- P_1(x,y))(r_2\xi_k- P_2(x,y))}.
	\end{equation}
	Evaluating \eqref{eq:2D-trap} in all the pairs of scaled $n$-th roots of unity yields
	\[
	\sum_{h,k=0}^{n-1} \frac{f(r_1\xi_h, r_2\xi_k)\cdot r_1r_2\xi_{h+k}}{{n^2}\cdot (r_1\xi_h- P_1(r_1\xi_s,r_2\xi_t))(r_2\xi_k- P_2(r_1\xi_s,r_2\xi_t))}-\rho f(r_1\xi_s,r_2\xi_t)\approx 0,\quad s,t=1,\dots,n.
	\]
	Rearranging the (approximate) evaluations of $f$ into the vector $\vc v_f\in\mathbb C^{n^2}$,  $(\vc v_f)_{h+nk} \approx f(r_1\xi_h,r_2\xi_k)$, leads us to the eigenvalue problem
	\begin{equation}\label{eq:G2D3}
	A \vc{v_f}= \lambda_{min}\vc{v_f},\qquad A = C_{P_1,P_2,r_1,r_2}^{(n^2)} D-\rho I_{n^2}\in\mathbb C^{n^2\times n^2},
	\end{equation} 
	where \[
	\left(C_{P_1,P_2,r_1,r_2}^{(n^2)}\right)_{s+n(t-1),h+n(k-1)}:= \frac{1}{(r_1\xi_h- P_1(r_1\xi_s,r_2\xi_t))(r_2\xi_k- P_2(r_1\xi_s,r_2\xi_t))},
	\] $D=n^{-2}\diag(\vc{\xi}_{(1)}\otimes\vc{\xi}_{(2)})$, and $\vc{\xi}_{(j)}\in\mathbb C^n $ is the vector containing the $n$-th roots of unity in the counterclockwise order multiplied by the constant $r_j$.
	
	After computing a vector $\vc{v_f}$ that verifies \eqref{eq:G2D3}, we apply the two-dimensional FFT on it; for example, this task is performed by the MATLAB command \texttt{ifft2(reshape($\vc{v_f}$, n, n))}. This returns the (approximate) coefficients of the interpolating bivariate polynomial $\sum_{h,k=0}^{n-1}\widetilde f_{h,k}x^hy^k$ for $\widetilde f(x,y):=f(r_1 x,r_2 y)$.  In order to obtain those for $f(x,y)$ we rescale them with the rule $f_{h,k}\gets \widetilde f_{h,k}/(r_1^hr_2^k)$.  Once again, we impose the boundary condition  $0=G(0,0)=1+tf(0,0)=1+tf_{0,0}$, which implies $t=-{1}/{f_{0,0}}$. This yields the following (approximate) interpolating bivariate polynomial $\widehat G(x,y)$ for $G(x,y)$:
	\[
	\widehat G(x,y):=\sum_{h,k=0}^{n-1}\widehat g_{h,k}x^hy^k,\qquad \begin{cases}\widehat g_{0,0}=0\\ \widehat g_{h,k}= -\frac{f_{h,k}}{f_{0,0}},&(h,k)\neq(0,0).\end{cases}
	\]
	The whole procedure is summarized in Algorithm~\ref{alg:ev-int-2D}. 
	
	In our implementation, the parameters $r_1$ and $r_2$ are set as $r_j=\arg\min_{x\geq 1} P_j(x,x)-x$, $j=1,2$. A significant difference in the parameters $r_1$ and $r_2$ suggests the use of different levels of discretization on the two integrals. This requires to slightly modify Algorithm~\ref{alg:ev-int-2D} in order to consider $n_1$ and $n_2$ integration nodes for the two integrals in \eqref{eq:2D-trap}. In the numerical experiments of Section~\ref{sec:2d-lin-frac} we always use $n_1=n_2=n$.

	\begin{algorithm}
		\caption{Evaluation-Interpolation in the 2D case}\label{alg:ev-int-2D}
		\begin{algorithmic}[1]
			\Procedure{Compute\_G\_2D}{$P_1(x,y),r_1, P_2(x,y),r_2, n$}\Comment{$r_j\in(1,\psi_{P_j})$}
			%		\State $\rho \gets$ spectral radius of $\begin{bmatrix}
			%		\frac{\partial P_1(1,1)}{\partial x}&\frac{\partial P_1(1,1)}{\partial y}\\
			%		\frac{\partial P_2(1,1)}{\partial x}&\frac{\partial P_2(1,1)}{\partial y}
			%		\end{bmatrix}$
			\State $\rho \gets$ spectral radius of $M$ given by \eqref{MM}
			\State $\vc{\xi}\gets\left(e^{\frac{2\pi \vc i j}{n}}\right)_{j=1,\dots, n}$
			\State $A\gets\left(\frac{1}{(r_1\xi_h- P_1(r_1\xi_s,r_2\xi_t))(r_2\xi_k- P_2(r_1\xi_s,r_2\xi_t))}\right)_{s+n(t-1),h+n(k-1)},\quad h,k,s,t,=1,\dots, n$
			\State $D\gets \frac{r_1r_2}{n^2}\diag(\xi\otimes\xi)$
			\State $A\gets A \cdot D-\rho\cdot I_{n^2}$
			\State $\vc v_f\gets$ \Call{Eigs}{$A$},\qquad $V_f\gets$ \Call{Reshape}{$\vc v_f,n,n$}
			\State $\widehat G\gets$ \Call{IFFT2}{$V_f$},\quad $\widehat G\gets \left( \frac{\widehat g_{h,k}}{r_1^{h}r_2^{k}}\right)_{h,k=0,\dots,n-1}$\Comment{$\sum_{h,k=0}^{n-1}\widehat g_{h,k}x^{h}y^{k}$ interpolates $f(r_1x,r_2y)$}
			\State $\widehat G\gets -\frac{1}{\widehat g_{0,0}} \widehat G$,\quad $\widehat g_{0,0}\gets0$ 
			\State \Return $ \widehat G$
			\EndProcedure
		\end{algorithmic}
	\end{algorithm}
	\subsection{Rank structure in the matrix $A$}\label{sec:rank-2D}
	The size of the linear system \eqref{eq:G2D3} depends quadratically on the number of nodes $n$ that we use for discretizing each integral in \eqref{eq:G2D2}. In particular, the execution --- in dense arithmetic --- of Algorithm~\ref{alg:ev-int-2D} rapidly become computationally not feasible as $n$ increases. We here see that, similar to the single-type scenario, the matrix $A$ exhibits a rank structure and we discuss how to modify Algorithm~\ref{alg:ev-int-2D}  in order to consider larger values of $n$.
	
	Let us  denote with $\vc 1_n$ the vector of all ones of length $n$; then we can write
	\[
	C_{P_1,P_2,r_1,r_2}^{(n^2)}=C(\vc 1_n\otimes \vc \xi_{(1)}, P_1(\vc \xi_{(1)},\vc \xi_{(2)}))\circ C( \vc \xi_{(1)}\otimes \vc {1}_n, P_1(\vc \xi_{(1)},\vc \xi_{(2)})),
	\]
	where $\circ$ denotes the Hadamard product.
	In light of \eqref{eq:G2D3}, $A$ is obtained by applying a column scaling and a diagonal shift to the Hadamard product of two Cauchy matrices. Intuitively, if the latter are both  numerically low-rank, then we expect the numerical rank of $C_{P_1,P_2,r_1,r_2}^{(n^2)}$ to be much smaller than $n^2$. More formally, as pointed out in \cite[Section 4.2]{town}, Hadamard products of Cauchy matrices solve certain rank structured linear matrix equations and this enables to state decaying bounds for their singular values; see Theorem~2 in \cite{town}. 
	
	Since we have a cheap access to the entries of $C_{P_1,P_2,r_1,r_2}^{(n^2)}$, we compress it using Algorithm~\ref{alg:aca} in place of forming the full matrix $A$ in Line 4 of Algorithm~\ref{alg:ev-int-2D}. Once again, this provides two tall and skinny matrices $U,V$ whose storage consumption is $\mathcal O(n^2)$. Finally, we apply the diagonal scaling to the matrix $V$ and we replace the call to \textsc{Eigs}, in Line 7, with a call to Algorithm~\ref{alg:eigs-lr}. The modified procedure is reported in Algorithm~\ref{alg:lr-ev-int-2D} and tested in the next example.
	\begin{algorithm}
		\caption{Low-rank Evaluation-Interpolation in the 2D case}\label{alg:lr-ev-int-2D}
		\begin{algorithmic}[1]
			\Procedure{Compute\_G\_2D\_LR}{$P_1(x,y),r_1, P_2(x,y),r_2, n$}
			\State $\rho \gets$ spectral radius of $M$ given by \eqref{MM}
			\State $\vc{\xi}\gets\left(e^{\frac{2\pi \vc i j}{n}}\right)_{j=1,\dots, n}$
			\State $[U,V]\gets \Call{ACA}{C_{P_1,P_2,r_1,r_2}^{(n^2)}}$
			\State $D\gets \frac{r_1r_2}{n^2}\diag(\xi\otimes\xi)$
			\State $V\gets V \cdot D$
			\State $\vc v_f\gets$ \Call{Eigs\_LR}{$U,V,\rho$},\qquad $V_f\gets$ \Call{Reshape}{$\vc v_f,n,n$}
			\State $\widehat G\gets$ \Call{IFFT2}{$V_f$},\quad $\widehat G\gets \left( \frac{\widehat g_{h,k}}{r_1^{h}r_2^{k}}\right)_{h,k=0,\dots,n-1}$
			\State $\widehat G\gets -\frac{1}{\widehat g_{0,0}} \widehat G$,\quad $\widehat g_{0,0}\gets0$ 
			\State \Return $ \widehat G$
			\EndProcedure
		\end{algorithmic}
	\end{algorithm}
	
	\subsection{Numerical tests}
	\begin{example}\label{ex:2D-linfrac}
		We first consider a two-type linear fractional GW process whose offspring distribution is defined in Section~\ref{sec:2d-lin-frac}, with the following parameters:
		\[
		S=\begin{bmatrix}
		0.3&0.2\\
		0.2&0.3
		\end{bmatrix},\qquad \vc c=\begin{bmatrix}
		-0.05\\
		-0.05
		\end{bmatrix},\qquad
		\vc b=\begin{bmatrix}
		0.4\\
		0.4
		\end{bmatrix},
		\qquad
		d=1.
		\]
		Via direct computation we find that $\rho=\frac 23$, and the explicit expression of the coefficients $g_{h,k}$ is given in \eqref{eq:ghk} with $\nu_1=\nu_2= \frac 12$ and $\mu_1=\mu_2=\frac 18$. 
		
		We run Algorithm~\ref{alg:lr-ev-int-2D} for this example with  $n=256$, and we compute the relative error $|(\widehat g_{h,k}-g_{h,k})/g_{h,k}|$ of the approximate coefficients $\widehat g_{h,k}$ returned by our method. The results are shown in Figure~\ref{fig:2D-err} where we let the indices $h,k$ vary in $[0,20]$  (outside this range, the coefficients $g_{h,k}$ are below the machine precision). We see that the most accurate quantities are those with highest magnitude, i.e., the coefficients $\widehat g_{h,k}$ whose index $(h,k)$ is close to $(0,0)$. The relative error increases progressively as $h,k$ increase, reaching about $10^{-4}$ for the quantities that are at the level of the machine precision.
	\end{example}
	\begin{figure}\label{fig:2D-err}
		\centering
		\begin{tikzpicture}
		\begin{axis}[width=.55\linewidth, height=.27\textheight, zmode = log, view={60}{30},colorbar,
		colorbar style={ytick={-13.8155,  -18.4207,  -23.0259,  -27.6310},yticklabels={$10^{-6}$,$10^{-8}$,$10^{-10}$, $10^{-12}$}}]
		\addplot3[surf, z buffer=sort,scatter,only marks, opacity=0.6, color = blue] file {e_2D_linfrac.dat};  
		\end{axis}
		\end{tikzpicture}\caption{Example~\ref{ex:2D-linfrac}. Relative error of the computed coefficients $\widehat g_{h,k}$ in the bivariate linear fractional example.}
	\end{figure}
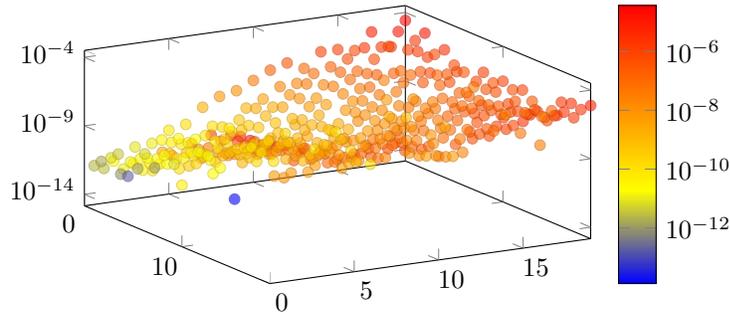
	\begin{example}\label{ex:3}
		Here we test the scalability of Algorithm~\ref{alg:lr-ev-int-2D} on a randomly generated example.
		We consider $P_1(x,y)$ and $P_2(x,y)$ equal to bivariate polynomials of degree $(2,2)$ with  coefficients reported in Table~\ref{tab:coef}. This yields $\rho\approx 0.5884$, $r_1 = 1.2462$ and $r_2=1.4104$. The radii $r_j$ are estimated using the MATLAB function \texttt{fminsearch}.

		In Figure~\ref{fig:e-2D} left, we show the performances of Algorithm~\ref{alg:lr-ev-int-2D} and the features of the computed solution. For all values of $n$ the computed coefficients $\widehat g_{h,k}$ are non negative up to machine precision. 
		The reported residual is defined as \[\text{Res}:=\max_{i,j=1,\dots,n}|\widehat G(P_1(\xi_i,\xi_j), P_2(\xi_i,\xi_j))-\rho\widehat G(\xi_i,\xi_j)-1+\rho|,\qquad \xi_j = \exp(2\pi\vc i j/n).\]
		In the last column we also report the rank of the approximation of the matrix $C_{P_1,P_2,r_1,r_2}^{(n^2)}$ returned by Algorithm~\ref{alg:aca}.  The growth of this quantity --- as the number of nodes  increases --- makes the time consumption  slightly super-quadratic with respect to $n$. In Figure~\ref{fig:e-2D}-right, we plot the  coefficients $\widehat g_{h,k}$ up to degree $(63,63)$, computed in the case $n=512$. Experimentally, we observe that if $h$ or $k$ is not in the range $[0,63]$ then $\widehat g_{h,k}\leq 10^{-31}$. This confirms that a small number of coefficients is sufficient to describe the quasi-stationary distribution with high accuracy.
	\end{example}
	\begin{table}\label{tab:coef}
		\centering
		\begin{tabular}{l|lllllllll}
			&$p_{0,0}^{(j)}$  &$p_{0,1}^{(j)}$ &$p_{0,2}^{(j)}$ &$p_{1,0}^{(j)}$ &$p_{1,1}^{(j)}$ &$p_{1,2}^{(j)}$ &$p_{2,0}^{(j)}$ &$p_{2,1}^{(j)}$ &$p_{2,2}^{(j)}$   \\ \hline
			$P_1(x,y)$&$0.798$&$ 0.029$&$0.009$&$0.015$&$0.010$&$0.022$&$ 0.052$&$0.020$&$0.045$  \\ 
			$P_2(x,y)$&$0.694$&$ 0.041$&$0.057$&$0.035$&$0.027$&$0.043$&$0.024$&$0.051$&$0.028$  \\
			\hline
		\end{tabular}
		\caption{Example~\ref{ex:3}. Coefficients of the bivariate polynomials $P_j(x,y)$.}
	\end{table}
	
	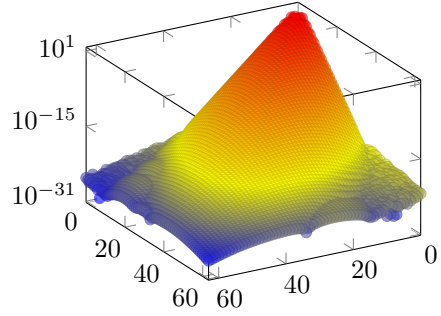
\begin{figure}
		\centering
		\begin{minipage}{0.55\textwidth}
			\begin{filecontents*}{e_2D.dat}
16	0.11684	0.26628		-0.0038163	1.1699	107
32	0.17613	0.065917	0		0.89742	192
64	0.85721	0.0022557	0		0.99725	306
128	5.6621	1.4502e-05	0		0.99998	437
256	33.447	9.5251e-10	-5.3249e-21	1	572
512	218.49	5.9493e-12	-8.8356e-21	1	673
\end{filecontents*}
\pgfplotstabletypeset[%
			sci zerofill,
			columns={0,1,2,4, 5},
			columns/0/.style={column name=$n$},
			columns/1/.style={column name=Time (s), fixed},
			columns/2/.style={column name=Res},
			%columns/3/.style={column name=$\min g_j$},
			columns/4/.style={column name=$\sum \widehat g_{h,k}$},
			columns/5/.style={column name=rank}
			]{e_2D.dat}
		\end{minipage}~\begin{minipage}{0.42\textwidth}
			\bigskip 
			
			\begin{tikzpicture}
			\begin{axis}[width=.95\linewidth, height=.27\textheight, zmode = log, view={-60}{-30}]
			\addplot3[surf, z buffer=sort,scatter,only marks, opacity=0.6, color = blue] file {e_2D_G.dat};   
			\end{axis}
			\end{tikzpicture}
		\end{minipage}
		\caption{Example~\ref{ex:3}. On the left, performances of Algorithm~\ref{alg:ev-int-2D} as  $n$ increases. On the right, 2D plot of the coefficients $\widehat g_{h,k}$, in the case $n=512$.}
		\label{fig:e-2D}
	\end{figure}

	\section{Conclusions}
	We provided a fully algebraic analysis of the interplay between the regularity of the offspring distribution and that of the quasi-stationary distribution of a subcritical GW process.
	We proposed a new numerical method for computing the quasi-stationary distribution. We showed that our approach can significantly outperform the accuracy of other techniques based on simulations or on interpolation. 
	
	Moreover, we provided a theoretical analysis of the low-rank structure stemming from the discretization of the problem. This enabled our algorithm to be slightly modified in order to scale well with the fineness of the discretization. The reported  numerical tests confirm the scalability of computational time.
\section{Appendix}
Here we  report the proofs of some of the results in Section~\ref{sec:properties}.
\begin{proof}[Proof of Proposition~\ref{prop:1}]
	First, note that $G(z)$ analytic on $\mathcal B(0,r_G)$ with $r_G>1$ implies $G(z)$ analytic at $1$.
	
	Now, let us assume $G(z)$ analytic on an open neighborhood $A_1$ of $1$. We proceed by proving that $G(z)$ is analytic at every point of $\mathcal S^1$. Given $z\in\mathcal S^1$ we distinguish between two cases: $|P(z)|<1$ and $|P(z)|=1$. If $|P(z)|<1$, then there exists an open neighborhood $A_z$ of $z$  such that $|P(\widetilde z)|<1$ $\forall \widetilde z\in A_z$. Since $G(z)$ verifies \eqref{eq:G}, the expression 
	\begin{equation}\label{eq:continuation}
	G(z)=m^{-1}(G(P(z)) +m -1)
	\end{equation}
	provides an analytic continuation of $G(z)$ on  $A_z$. 
	If $|P(z)|=1$, then  $1=|p_0+w|$ where $w:=\sum_{j=1}^\infty p_jz^j$.
	%	p_0 +\underbrace{\sum_{j=1}^\infty p_jz^j}_{w}|=|p_0+w|
	%	\[
	%	1=|p_0 +\underbrace{\sum_{j=1}^\infty p_jz^j}_{w}|=|p_0+w|.
	%	\]
	Since $p_0\in(0,1)$ and $|w|\leq \sum_{j=1}^\infty p_j=1-p_0$, the sum $p_0+w$ has modulus $1$ if and only if $w=1-p_0$, i.e., $P(z)=1\in A_1$. In particular, there exists an open neighborhood $A_z$ of $z$ such that  $P(\widetilde z)\in A_1$ $\forall \widetilde z\in A_z$. Once again, \eqref{eq:continuation} defines an analytic continuation of $G(z)$ on $A_z$.
	
	By construction, $G(z)$ is analytic on $\mathcal B(0,1)$ and the union  $\mathcal B(0,1)\bigcup \{A_z\}_{z\in \mathcal S^1}$ yields an open set $\widehat A$ that contains $\mathcal D(0,1)$ where $G(z)$ is analytic.
	%forms a covering of $\mathcal D(0,1)$. The compactness of  $\mathcal D(0,1)$ ensures that $\exists  z_1,\dots,z_k$  such that $\mathcal D(0,1)\subset\widehat A:=\mathcal B(0,1)\bigcup_{j=1}^k \{A_{z_j}\}$ and $G(z)$ is analytic on $\widehat A_k$. 
	This implies that there exists $r_G>1$ such that $B(0,r_G)\subseteq \widehat A$ and $G(z)$ is analytic on $\mathcal B(0, r_G)$.
\end{proof}
\subsection{Differentiability of $G(z)$}
First we establish the relationship between the higher order derivatives of $G(z)$ and those of $P(z)$.
Differentiating \eqref{eq:G} $h$ times ($h\geq 1$) leads to \begin{equation}\label{hder}m\,G^{(h)}(z)= (G\circ P)^{(h)}(z),\end{equation} where $(G\circ P)(z):=G(P(z))$. The derivative of the composition is expressed in closed form with the \emph{Fa\`a di Bruno's formula} \cite{comtet},
\begin{equation}\label{eq:dibruno}
(G \circ P)^{(h)}(z)= \sum_{j=1}^hG^{(j)}(P(z))\cdot B_{h,j}\left(P^{(1)}(z),\dots,P^{(h-j+1)}(z)\right),
\end{equation}
which involves the  so-called \emph{Bell polynomials} $B_{h,j}$ \cite{miho}, defined as 
\[
B_{h,j}(x_1,\dots,x_{h-j+1}):=\sum \frac{h!}{j_1!\dots j_{h-j+1}!}\prod_{s=1}^{h-j+1}\left(\frac{x_s}{s!}\right)^{j_s},
\]
where the sum is taken over all sequences $j_1,j_2,\dots,j_{h-j+1}$ of non-negative integers such that
$\sum_{s=1}^{h-j+1}j_s=j$ and $\sum_{s=1}^{h-j+1}s\cdot j_s=h$. In particular we have $B_{h,h}(x_1)=x_1^h$. Plugging \eqref{eq:dibruno} into \eqref{hder} and evaluating at $z=1$ yields the relation
\begin{equation}\label{eq:derivative}
(m-m^h)G^{(h)}(1)=\sum_{j=1}^{h-1}G^{(j)}(1)\cdot B_{h,j}\left(P^{(1)}(1),\dots,P^{(h-j+1)}(1)\right),
\end{equation} which is only informative for $h>1$.
Equation \eqref{eq:derivative} highlights the connection between the existence of higher order derivatives of $P(z)$ and those of $G(z)$. In probabilistic terms, it relates the factorial moments of the quasi-stationary distribution to those of the offspring distribution. We are now ready to prove the following lemma.

\begin{lemma}\label{lem:deriv}
	For any $1<h\in\mathbb N$, $P^{(h)}(1)$ is finite if and only if $G^{(h)}(1)$ is finite.
\end{lemma}
\begin{proof}
	First, assume that $P^{(h)}(1)<\infty$; observe that this implies $\sum_{j=2}^\infty j\log(j)p_j<\infty$, or equivalently, $G^{(1)}(1)<\infty$, in light of Theorem~\ref{thm:heath}.
	Moreover, \eqref{eq:derivative} expresses $G^{(h)}(1)$  as a linear combination of $G^{(1)}(1),\dots,G^{(h-1)}(1)$, whose coefficients are polynomial functions of  $P^{(1)}(1),\dots,$ $P^{(h)}(1)$. The claim then follows using an inductive argument.
	
	Next, assume $G^{(h)}(1)<\infty$. The only term which involves the $h$-th derivative of $P(z)$ in the right-hand side of \eqref{eq:derivative} is $G^{(1)}(1)P^{(h)}(1)$, which is obtained by choosing $j=1$, $j_1=\dots,j_{h-1}=0$ and $j_h=1$ in the series expansion. Since $G^{(1)}(1)\neq 0$, this allows us to express  $P^{(h)}(1)$ as a well-defined function of $G^{(1)}(1),\dots,G^{(h)}(1)$ and  $P^{(1)}(1),\dots,P^{(h-1)}(1)$. The claim then again follows by induction.
\end{proof}

\subsection{Proof of Theorem~\ref{thm:radius}}
We recall an identity regarding Bell polynomials.
\begin{lemma}[Wang and Wang \cite{wang}, Lemma 2.6]\label{lem:bell-id}
	Let $f(z):=\sum_{j=1}^{\infty}\frac{f_j}{j!} z^j$, then $\forall h,k\in\mathbb N$,
	\[
	B_{h,k}(f_1,\dots,f_{h-k+1})=\frac{h!}{k!}\cdot [z^h](f(z)^k),
	\] 
	where $[z^h](\cdot)$ denotes the operator that extracts the $h$-th coefficient from the power series expansion of the argument around zero.
\end{lemma} 

\begin{proof}[Proof of Theorem~\ref{thm:radius}]
	
	The function $P(z)$ being analytic at $z=1$ by assumption, we consider the power series expansion of $\widetilde P(z):=P(1+z)=\sum_{j\geq 0}\tilde p_jz^{j}$ that has radius of convergence $r_P-1$.
	
	By Proposition~\ref{prop:1}, the claim is equivalent to having $G(z)$ analytic at $1$. Therefore, we proceed by considering the (left looking) Taylor expansion of $G(z)$ at $1$ and by showing that its radius of convergence is non-zero. In view of \eqref{eq:exp-decay} this is equivalent to showing that $\exists \rho,\theta_G >0$  such that 
	\[
	G^{(h)}(1)\leq \theta_G \cdot \rho^{-h}\cdot h!,\qquad \forall h\in\mathbb N.
	\]
	Choosing $\theta_G=\max\{1,G^{(1)}(1)\cdot \rho\}$ provides the claim for $h=0$ and $h=1$ without limiting the parameter $\rho$. For $h>1$ we use an inductive argument; from \eqref{eq:derivative} we get
	\begin{align*}
	G^{(h)}(1)&=\frac{1}{m-m^h}\sum_{j=1}^{h-1}G^{(j)}(1)\cdot B_{h,j}\left(P^{(1)}(1),\dots,P^{(h-j+1)}(1)\right)\\
	&\leq \frac{\theta_G}{m-m^h}\sum_{j=1}^{h-1}\rho^{-j}\cdot  j!\cdot B_{h,j}\left(P^{(1)}(1),\dots,P^{(h-j+1)}(1)\right).
	\end{align*}
	Observe that $P^{(j)}(1)=j!\cdot \tilde p_j$, therefore, Lemma~\ref{lem:bell-id} implies
	\[
	B_{h,j}\left(P^{(1)}(1),\dots,P^{(h-j+1)}(1)\right) =\frac{h!}{j!}\cdot [z^h]\left((\widetilde P(z)-1)^j\right).
	\]
	Since $(\widetilde P(z)-1)^j$ also has radius of convergence $r_P-1$ and its expansion involves non-negative coefficients, the $h$-th coefficient of the latter satisfies
	\[
	[z^h]\left((\widetilde P(z)-1)^j\right)\leq (\widetilde P(r)-1)^jr^{-h},\qquad \forall r\in(0, r_P-1).
	\]
	In particular, selecting $\widetilde r\in(0,\psi_P-1)$, where $\psi_P$ is given in \eqref{psiP}, provides $P(1+\widetilde r)\leq  1+\widetilde r$ and $
	[z^h]\left((\widetilde P(z)-1)^j\right)\leq \widetilde r^{(j-h)}.$ 
	Coming back to $G^{(h)}(1)$, we then have
	\[
	G^{(h)}(1)\leq \frac{\theta_G\cdot h!}{\widetilde r^h(m-m^h)}\sum_{j=1}^{h-1}\left(\frac{\widetilde r}{\rho}\right)^{j}= \theta_G\cdot \rho^{-h}\cdot h! \cdot\frac{ \left(\frac{\widetilde r}{\rho}\right)^{1-h}-1}{\left(1-\frac{\widetilde r}{\rho}\right)(m-m^h)}. 
	\]
	Choosing $\rho$ small enough we can ensure $\frac{ \left(\frac{\widetilde r}{\rho}\right)^{1-h}-1}{\left(1-\frac{\widetilde r}{\rho}\right)(m-m^h)}<1$  independently of $h>1$. This completes the proof.
\end{proof}
	\section{Acknowledgements}
	Sophie Hautphenne thanks the Australian Research Council for support through Discovery Early Career Researcher Award DE150101044. The authors also thank Daniel Kressner, Beatrice Meini and Phil Pollett for fruitful discussions, as well as three anonymous referees for valuable comments and suggestions.
	\bibliography{biblio}
	\bibliographystyle{siamplain}
	
	\end{document}